\documentclass{article}

\usepackage[english]{babel}
\usepackage{amsthm}
\usepackage{amsmath}
\usepackage{amssymb}
\usepackage{graphicx}
\usepackage{enumitem}
\usepackage[colorlinks]{hyperref}
\usepackage{caption}
\usepackage{subcaption}
\usepackage{pict2e,color,colordvi,amscd}
\usepackage{indentfirst}
\usepackage{tikz}
\usepackage{float}

\newtheorem{definition}{Definition}[section]
\newtheorem{theorem}[definition]{Theorem}
\newtheorem{lemma}[definition]{Lemma}
\newtheorem{claim}[definition]{Claim}
\newtheorem{corollary}[definition]{Corollary}

\newtheorem{proposition}[definition]{Proposition}

\newtheorem{observation}{Observation}

\newcommand{\bind}{\operatorname{bind}}

\title{Strong binding numbers and factors}
\date{\relax}
\author{
    Guantao Chen\thanks{Department of Mathematics and Statistics, Georgia State University, Atlanta, GA 30303, gchen@gsu.edu. Research of this author was partially supported by NSF grant DMS-2154331.} \and
    Mikhail Lavrov\thanks{misha.p.l@gmail.com.} \and
    Yuying Ma\thanks{Department of Mathematics and Statistics, Georgia State University, Atlanta, GA 30303, yma29@gsu.edu.} 
    \and Jennifer Vandenbussche\thanks{Department of Mathematics, Kennesaw State University, Marietta, GA 30060, jvandenb@kennesaw.edu.} 
    \and Hein van der Holst \thanks{Department of Mathematics and Statistics, Georgia State University, Atlanta, GA 30303, hvanderholst@gsu.edu.} 
}

\begin{document}
\maketitle

\begin{abstract}
Let $G$ be a simple graph. The \textit{$k$-th neighborhood} of a vertex subset $S \subseteq V(G)$, denoted $\Lambda^k(S)$, is the set of vertices that are adjacent to at least $k$ vertices in $S$. The \textit{$k$-th binding number} $\beta^k(G)$ is defined as the minimum ratio $|\Lambda^k(S)|/|S|$ over all subsets $S \subseteq V(G)$ with $|S| \ge k$ and $\Lambda^k(S) \ne V(G)$. This parameter generalizes the classical binding number introduced by Woodall.
Andersen showed that the condition $\beta^1(G) \ge 1$ does not guarantee the existence of a $1$-factor in $G$, while Bar\'at \textit{et al.} proved that $\beta^2(G) \ge 1$ suffices for the existence of a $2$-factor.  In this paper, we extend this result to general $k \ge 2$ by showing that any graph $G$ with even $k|V(G)|$ and $\beta^k(G) \ge 1$ contains a $k$-factor. Moreover, if $G$ is additionally a split graph of even order, then it admits a $(k+1)$-factor. We also prove that any graph $G$ with $\beta^k(G) \ge 1$ contains at least $k-1$ disjoint perfect or near-perfect matchings. Finally, for any bipartite graph $G$ with bipartition $(X, Y)$, we introduce an analogue of the $k$-th binding number and show that, under the condition $\beta^k(G, X) \ge 1$, the graph admits $k$ disjoint matchings, each covering $X$.

\vspace{0.3cm}
\end{abstract}

\section{Introduction}\label{sec:introduction}

All graphs considered in this paper are simple, that is, finite and without loops or multiple edges. Let $G$ be a graph with vertex set $V(G)$ and edge set $E(G)$. The {\em order} of $G$, denoted by $|G|$, is the number of vertices in $G$, that is, $|G| = |V(G)|$. For a subset $S \subseteq V(G)$, denote by $G[S]$ the subgraph induced by $S$, and denote by $G-S$ the subgraph induced by $V(G)\setminus S$. The {\em neighborhood} of $S$, denoted $\Lambda_G(S)$, is the set of vertices in $V(G)$ that are adjacent to at least one vertex in $S$. 
More generally, we define the {\em $k$-th neighborhood} of $S$, $\Lambda^k_G(S)$, as the set of vertices in $V(G)$ that have at least $k$ neighbors in $S$.
In particular, for $S = \{v\}$ we write $\Lambda_G(v)$ for $\Lambda_G(\{v\})$ and define the degree $d_G(v)$ by $d_G(v) = |\Lambda_G(v)|$.
The maximum and minimum degrees of $G$ are denoted by $\Delta(G)$ and $\delta(G)$, respectively. For a positive integer $k$, a graph $G$ is said to be {\em $k$-regular} if $\Delta(G) = \delta(G) = k$. A {\em $k$-factor} of a graph $G$ is a $k$-regular spanning subgraph.  In particular, a $1$-factor is also known as a {\em perfect matching}. When $G$ has odd order, a {\em near-perfect matching} is a $1$-regular subgraph containing all vertices of $G$ except one. 

The {\em binding number} of a graph $G$, introduced by Woodall~\cite{WOODALL1973}, is defined as 
\[
\bind(G) = \min\left\{\frac{|\Lambda_G(S)|}{|S|}: \emptyset \subset S \subseteq V(G), \Lambda_G(S) \ne V(G)\right\}.
\]
In other words, $\bind(G) \ge c$ implies that every subset $S \subseteq V(G)$ satisfies either $|\Lambda(S)| \ge c|S|$ or $\Lambda(S) = V(G)$. The relationship between the binding number and the existence of $k$-factors has been extensively studied. Anderson~\cite{ANDERSON1971} proved that every graph $G$ of even order with $\bind(G) \ge \frac 43$ contains a $1$-factor, and the bound $\frac 43$ is best possible, as demonstrated by the graphs $(r+2)K_3 + rK_1$, where $r \in \mathbb{Z}^+$ and $+$ denotes the {\em complete join} of two vertex-disjoint graphs — adding all edges between the two parts. In the same paper where the binding number was introduced, Woodall~\cite{WOODALL1973} showed that any graph with $\bind(G) \ge \frac 32$ has a Hamiltonian cycle. Later, Katerinis and Woodall~\cite{KATERINISWOODALL} extended the result to general $k \ge 2$, showing that every graph of order at least $4k-6$ with $nk$ even and $\bind(G) > \frac{(2k-1)(n-1)}{kn-2k+3}$ contains a $k$-factor. They also showed that this bound is sharp in $K_{2l-2r} + lK_2$, where $r \in \mathbb{Z}^+$ and $l = rk-1$.

Andersen's construction~\cite{ANDERSON1971} demonstrates that the condition $|\Lambda(S)| \ge |S|$ for all nonempty subsets $S \subseteq V(G)$ does not necessarily guarantee the existence of a $1$-factor in a graph $G$ of even order, while Bar\'at \textit{et al.}~\cite{Hungarian} proved that any graph $G$ of order at least two that satisfies $|\Lambda^2(S)| \ge |S|$ for all subsets $S \subseteq V(G)$ with $|S| \ge 2$ must contain a $2$-factor.
This contrast motivates us to study whether vertex expansion conditions can be used to ensure the existence of factors of higher degree.
During our investigation, we found that while the classical binding number behaves rather differently, the higher-order versions exhibit strikingly consistent structural behavior. To capture this, we introduce a generalization of the classical binding number, called the {\em $k$-th binding number}~$\beta^k(G)$. Let $k$ be a positive integer. For a graph $G$, we define
\[ 
\beta^k(G) = \min\left\{\frac{|\Lambda^k_G(S)|}{|S|}: S \subseteq V(G), |S| \ge k, \Lambda^k_G(S) \ne V(G)\right\}.
\]
For completeness, we set $\beta^k(G) := 0$ when $|G| < k$. 
The following is our first main result.
\begin{theorem}\label{thm:k-factor}
Let $k \ge 2$ be an integer, and let $G$ be a graph on $n$ vertices such that $nk$ is even. If $\beta^k(G) \ge 1$, then $G$ contains a $k$-factor. Moreover, the lower bound $1$ is sharp. 
\end{theorem}
We will show in Section~\ref{sec:properties} that $\beta^k(G) \ge 1$ implies that $G$ is non-bipartite and $n \ge 2k$.
We give two proofs of Theorem~\ref{thm:k-factor} (Sections~\ref{sec:first-proof} and~\ref{sec:k-factor}) and construct a family of examples (Section~\ref{sec:first-proof}) showing that the condition $\beta^k(G) \ge 1$ is tight. The first proof is relatively short and relies on the base case $k = 2$ together with a technical lemma introduced in Section~\ref{subsec:1-factors-nonbip}. The proof is concise, but the method seems difficult to extend—that is, to show that $\beta^k(G) \ge 1$ implies the existence of a $t$-factor in $G$ for some $t > k$. The second is based on Tutte's $k$-factor theorem, using maxmin barrier techniques. Although the proof is more involved, the underlying approach appears more likely to lead to a potential improvement in Theorem~\ref{thm:k-factor}, which we demonstrate via split graphs in Section~\ref{sec:split}.

A graph $G$ is called {\em a split graph} if its vertex set $V(G)$ can be partitioned into two disjoint sets $X$ and $Y$, where $X$ is an independent set and $Y$ is a clique. We refer to such a pair $(X, Y)$ as a {\em split partition}. We prove the following result:

\begin{theorem}\label{thm:k+1-factor-split-graph}
    Let $k \ge 2$ be an integer, and let $G$ be a split graph with split partition $(X, Y)$, where $|X|+|Y|$ is even. If $\beta^k(G) \ge 1$, then $G$ contains a $(k+1)$-factor.
\end{theorem}

Petersen~\cite{Petersen1891} proved that every $2k$-regular graph can be decomposed into $k$ edge-disjoint $2$-factors. However, the $k$-factor guaranteed by Theorem~\ref{thm:k-factor} does not necessarily decompose into $k$ disjoint $1$-factors; showing $\beta^k(G) \ge 1$ implies the existence of $k$ disjoint 1-factors would be a stronger result.  Instead, we prove something slightly weaker.

\begin{theorem}\label{thm:1-factors-nonbip}
Let $k \ge 1$ be an integer, and let $G$ be a graph on $n$ vertices. If $\beta^k(G) \ge 1$, then $G$ contains at least $k-1$ pairwise disjoint perfect matchings if $n$ is even, or $k-1$ near-perfect matchings if $n$ is odd.
\end{theorem}

In the context of bipartite graphs, a natural analogue of the binding number considers subsets from both sides of the bipartition. 
For a bipartite graph $G$ with bipartition $(X, Y)$, the {\em bipartite binding number} $\bind'(G)$ is defined as follows. If $G$ is complete bipartite, then $\bind'(G) = \min\left\{|X|, |Y|\right\}$. Otherwise, 
\[
\bind'(G) = 
\min\left\{
\displaystyle
\min_{\substack{\emptyset \ne S \subseteq X \\ \Lambda_G(S) \neq Y}} \frac{|\Lambda_G(S)|}{|S|},\quad
\min_{\substack{\emptyset \ne T \subseteq Y \\ \Lambda_G(T) \neq X}} \frac{|\Lambda_G(T)|}{|T|}
\right\}.
\]
This parameter has been studied in relation to factor problems in bipartite graphs. For instance, Hu \textit{et al.}~\cite{Hu2013} showed that if $|X| = |Y| > 139$ and $\bind'(G) > \frac{3}{2}$, then $G$ has a Hamiltionian cycle, and the bound $\frac{3}{2}$ is best possible. Another result due to Qian~\cite{Qian2001} shows that if $|X| = |Y| = n$ and $\bind'(G) > \frac{n-1}{2(\sqrt{kn+1}-k)}$, then $G$ contains a $k$-factor. 

While the bipartite binding number measures expansion by counting all neighbors, we introduce a refined, one-sided variant that considers only those with degree at least~$k$ within the selected subset:
\[
\beta^k(G, X) = \min\left\{ \frac{|\Lambda^k_G(S)|}{|S|} : S \subseteq X,\ |S| \ge k \right\},
\]
which we refer to as the \emph{weak bipartite $k$-th binding number}. For completeness, we define $\beta^k(G, X) := 0$ when $|X| < k$. The special case $\beta^2(G, X) \ge 1$ is known as the {\em double Hall property}; see~\cite{Hungarian, CLMSV2025} for more results.  Under this definition, we obtain the following result.
\begin{theorem} \label{thm:1-factor-bip}
Let $k \ge 1$ and $G$ be a bipartite graph with bipartition $(X, Y)$. If $\beta^k(G, X) \ge 1$, then $G$ contains $k$ disjoint matchings, each of which covers all vertices in $X$.
\end{theorem}

As an immediate consequence, if the bipartite graph $G$ is also balanced (i.e., $|X| = |Y|$), then it admits a $k$-factor.

The rest of the paper is organized as follows. In the next section, we discuss several structural properties of graphs related to the $k$-th binding number. Section~\ref{sec:1-factors} is devoted to the study of matchings. In Section~\ref{subsec:1-factors-nonbip}, we introduce a key technical lemma and use it to prove Theorem~\ref{thm:1-factors-nonbip} for non-bipartite graphs. In Section~\ref{subsec:1-factors-bip}, we prove Theorem~\ref{thm:1-factor-bip} for bipartite graphs. Leveraging the lemma, we then give a short proof of Theorem~\ref{thm:k-factor} in Section~\ref{sec:first-proof} and analyze the tightness of the lower bound. Sections~\ref{sec:k-factor} and~\ref{sec:split} are based on Tutte's $k$-factor theorem, so in Section~\ref{sec:Pre}, we review this theorem and give some technical results.
Section~\ref{sec:k-factor} presents the second proof of Theorem~\ref{thm:k-factor}, where we treat the cases of even and odd $k$ separately. Section~\ref{sec:split} contains the proof of Theorem~\ref{thm:k+1-factor-split-graph} for split graphs. We conclude the paper by outlining some directions for future research.

\section{Structural Properties of the \texorpdfstring{$k$}{k}-th Binding Number}\label{sec:properties}
We begin this section with some observations derived from the definition of the $k$-th binding number. If $G=(X, Y; E)$ is a bipartite graph with $|G| \ge k$, consider a vertex set $S$ of size $k$ intersecting both $X$ and $Y$. Since no vertex of $G$ can be adjacent to all of $S$, $\Lambda^k(S)=\emptyset$, and hence
 $\beta^k(G) =0$. We also note that for any graph $G$ with $|G| \ge k$ and for any $k$-set $S\subset V(G)$, $\Lambda^k(S)\cap S = \emptyset$. Hence $|G| < 2k$ implies $\beta^k(G) < 1$. 
 We generalize these two observations as follows.

\begin{observation}\label{obs-1}
Let $k \ge 2$ be an integer and $G$ be a graph of order $n$. If $\beta^k(G) > 0$, then $G$ is not bipartite. 
In addition, $\beta^k(G) \le \frac{n-k}{k}$, and consequently, if $\beta^k(G) \ge 1$, then $n \ge 2k$.
\end{observation}

We state a heritage property below.

\begin{observation}\label{obs:k_to_k-1}
    Let $k \ge 2$ be an integer. Every graph $G$ satisfies $\beta^i(G) \ge \beta^k(G)$ for each  $i \in [k]:= \{1,2, \dots, k\}$.
\end{observation}
\begin{proof} The statement is trivial if $\beta^k(G) = 0$ or $i = k$, so we assume $\beta^k(G) > 0$ and $i\in [k-1]$. By the definition of $\beta^k(G)$, we have $|G| \ge k > i$.
Let $S\subset V(G)$ with $|S| \ge i$ and $\Lambda^i(S) \ne V(G)$. If $|S| \ge k$, then $\Lambda^i(S) \supseteq \Lambda^k(S)$, so $|\Lambda^i(S)| \ge |\Lambda^k(S)| \ge \beta^k(G) |S|$.  
Hence, we may assume $|S| < k$. Let $S'\supset S$ with $|S'| =k$.   Since $|S'| =k$, $S' \subseteq \Lambda_G(v)$ for every $v\in \Lambda^k(S')$, which in turn gives $\Lambda^k(S') \subseteq \Lambda^i(S)$. Hence $|\Lambda^i(S)|\ge |\Lambda^k(S')| \ge \beta^k(G) |S'| \ge \beta^k(G) |S|$. 
\end{proof}

\begin{observation}\label{obs:connected-k-2}
  Every graph $G$ with $\beta^2(G) > 0$ is connected.
\end{observation}
\begin{proof}
    In a disconnected graph, we can pick a $2$-set $S$ intersecting two distinct components. By the choice of $S$, $\Lambda^2(S) = \emptyset$, which implies $\beta^2(G) = 0$.
\end{proof}

Combining Observations~\ref{obs:k_to_k-1} and~\ref{obs:connected-k-2}, we deduce the following:
\begin{corollary}\label{cor:connected}
    For every integer $k \ge 2$, every graph $G$ with $\beta^k(G) > 0$ is connected.    
\end{corollary}

However, for $k = 1$, the condition $\beta^1(G) > 0$ does not imply that $G$ is connected; it only ensures that $G$ has no isolated vertices. A specific example is the disjoint union of two complete graphs $K_s$ and $K_t$ with $s, t \ge 2$.

The next lemma gives a lower bound on $\delta(G)$, which will play a role in the subsequent sections.

\begin{lemma}\label{lem:minimum-degree-with-k}
For an integer $k \ge 2$ and for a graph $G$ of order $n$, let $\beta := \beta^k(G)$. Then the following statements hold.  

\begin{itemize}
    \item[(1)] If $\beta >0$, then 
    $\delta(G) \ge (\beta + 1)k - 1$, and

    \item[(2)] if $\beta \ge 1$, then
    $\delta(G) \ge \max\left\{(\beta + 1)k - 1, n-\dfrac{n-1}{\beta}\right\}$.

\end{itemize}
\end{lemma}
\begin{proof}

Suppose $\beta > 0$. Then $n \ge k$, and choose an arbitrary $v\in V(G)$. Let $S$ be a $k$-set containing $v$ and maximizing $|S \cap \Lambda(v)|$. That is, if $\Lambda(v) \setminus S \neq \emptyset$, then $S \subseteq \Lambda(v)\cup \{v\}$. Since $v \in S$, $|S| = k$ and $\beta > 0$, it follows that $\emptyset \neq \Lambda^k(S) \subseteq \Lambda(v)\setminus S$. Therefore $S \subseteq \Lambda(v) \cup \{v\}$, and we conclude that $|\Lambda(v)| \ge |S|-1+|\Lambda^k(S)| \ge k-1+\beta|S| = (\beta+1)k-1$. This proves $(1)$.

Suppose $\beta \ge 1$ and let $v\in V(G)$ be an arbitrary vertex. By Observation~\ref{obs-1}, $\beta \le \frac{n-k}k \le \frac{n-1}{k-1}$, and so $k-1\le \frac{n-1}{\beta}$. If $n - d(v) \le k-1$, then $d(v) \ge  n-(k-1) \ge n-\frac{n-1}{\beta}$. Suppose $n-d(v) \ge k$. Since $ \Lambda^k(V(G-\Lambda(v)))  \subseteq V(G-v)$, it follows that $n-1 \ge |\Lambda^k(V(G -\Lambda(v)))| \ge \beta (n -d(v))$, which in turn gives $d(v) \ge n -\frac{n-1}{\beta}$. So, $(2)$ holds. 
\end{proof}

Applying (1) of Lemma~\ref{lem:minimum-degree-with-k} with $\beta^k(G) \ge 1$, we get the following corollary. 
\begin{corollary}
    Let $k \ge 2$ be an integer. If $G$ is a graph on $2k$ vertices with $\beta^k(G) \ge 1$, then $G$ is complete.
\end{corollary}

The {\em independence number} $\alpha(G)$ is defined as the size of a largest independent set in a graph $G$, that is, a set of pairwise non-adjacent vertices. We have the following upper bound on $\alpha(G)$ in terms of $n$, $k$ and $\beta^k(G)$.

\begin{proposition}\label{lem:k-independence}
Let $k \ge 2$ be an integer, and let $G$ be a graph of order $n$ with $\beta^k(G) = \beta$. Then $\alpha(G) \le \frac{1}{\beta+1}n$. Moreover, if $\beta \ge 1$, then $\alpha(G) \le \frac{n-\beta(k-1)}{\beta+1}$.
\end{proposition}
\begin{proof}
If $\beta=0$, the result is clear, so suppose $\beta > 0$. Let $I$ be a maximum independent set in $G$. 

If $|I| \le k-1$, then $|I| <  k \le \frac{n}{\beta +1}$ as $\beta \le \frac{n-k}{k}$ by Observation~\ref{obs-1}.
If $|I| \ge k$, then $\Lambda^k(I) \subseteq V(G) \setminus I$, so $n - |I| \ge |\Lambda^k(I)| \ge \beta|I|$. Solving for $|I|$ yields $|I| \le \frac{n}{\beta+1}$. In both cases, $\alpha(G) = |I| \le \frac{1}{\beta+1}n$.   
    
Now suppose $\beta \ge 1$. We have $|I| \le \frac 1{\beta+1} n \le \frac n2$, and by Observation~\ref{obs-1}, $n \ge 2k$. So $|V(G)\setminus I| \ge \frac n2 \ge k$. Let $I' = I \cup W$, where $W \subseteq V(G) \setminus I$ satisfies $|W| = k-1$. Since $|I| \ge 1$, then $|I'| \ge k$. As each vertex in $I$ has no neighbors within $I$, it can have at most $k-1$ neighbors in $I'$. This implies that $\Lambda^k(I') \subseteq V(G) \setminus I$.
Hence $n - |I| \ge |\Lambda^k(I')| \ge \beta|I'| = \beta(|I|+k-1)$. Solving for $|I|$, we obtain $|I| \le \frac{n-\beta(k-1)}{\beta + 1}$.
\end{proof}

A graph $G$ is said to be {\em $t$-tough} if $|S| \ge t \cdot c(G-S)$ for any subset $S \subseteq V(G)$ with $c(G-S) > 1$, where $c(G-S)$ denotes the number of connected components of $G-S$. The {\em toughness} of $G$, denoted $\tau(G)$, is the maximum value of $t$ for which $G$ is $t$-tough. For completeness, it is customary to define $\tau(K_n) = \infty$. The following proposition relates the toughness of a graph to its $k$-th binding number.
\begin{proposition}\label{prop:toughness} 
For any integer $k \geq 2$, every graph $G$ satisfies $\tau(G) \ge \beta^k(G)$. 
\end{proposition}
\begin{proof}
     Since $\tau(G)$ is nonnegative, the inequality $\tau(G) \ge \beta^k(G)$ holds trivially if $\beta^k(G) = 0$. Assume $\beta^k(G) > 0$, and hence $|G| \ge k$. 
     It suffices to show that $\tau(G) \ge \beta^2(G)$ as $\beta^2(G) \ge \beta^k(G)$ by Observation~\ref{obs:k_to_k-1}.
     We may also assume that $G$ is neither disconnected nor complete, since otherwise the inequality $\tau(G) \ge \beta^2(G)$ follows immediately. Let $U \subseteq V(G)$ be a vertex set such that $G - U$ is disconnected, and let its components be $G_1, G_2, \dots, G_m$, where $m = c(G - U) \ge 2$. For each $i$, select a vertex $v_i \in V(G_i)$, and let $S = \{v_1, v_2, \dots, v_m\}$. By construction, $\Lambda^2(S) \subseteq U$, and thus $|U| \ge |\Lambda^2(S)| \ge \beta^2(G)|S| = \beta^2(G)m$. This shows that for every $U$ with $c(G - U) = m > 1$, we have $|U| \ge \beta^2(G) \cdot c(G - U)$, which implies $\tau(G) \ge \beta^2(G)$. 
\end{proof}

The {\em vertex connectivity} of a graph $G$, denoted $\kappa(G)$, is the minimum number of vertices whose removal either disconnects $G$ or reduces it to a single vertex. In general, a natural upper bound for $\kappa(G)$ is the minimum degree $\delta(G)$.
\begin{proposition}\label{prop:connectivity}
    Let $k \ge 2$ be an integer, and let $G$ be a graph of order $n$. If $\beta^k(G) \ge 1$, then $\kappa(G) \ge \frac{\beta^k(G) - 1}{\beta^k(G) + 1} n$. 
\end{proposition}
\begin{proof}
    The inequality holds trivially when $\beta^k(G) = 1$, so we may assume $\beta^k(G) > 1$.
    Let $U \subseteq V(G)$ be a minimum vertex set such that $G - U$ consists of a single vertex or is disconnected. We claim that in both cases, $|U| \ge \min\left\{\frac{\beta^k(G) - 1}{\beta^k(G) + 1} n, \delta(G)\right\}$. In the former case, $G$ is complete and $|U| = n-1 = \delta(G)$. In the latter case, the set $V(G) \setminus U$ can be partitioned into at least two nonempty subsets $S$ and $T$ with no edges between them. We consider two subcases. Suppose first that one of $S$ or $T$ has size $1$. Without loss of generality, suppose $S = \{v\}$. Then the vertex $v$ must be adjacent only to vertices in $U$, and thus $|U| \ge |\Lambda(v)| \ge \delta(G)$.
    In the remaining case, both $|S| \ge 2$ and $|T| \ge 2$. We have $\Lambda^2(S) \subseteq S \cup U$, so $|S|+|U| \ge |\Lambda^2(S)| \ge \beta^2(G)|S| \ge \beta^k(G)|S|$ by Observation~\ref{obs:k_to_k-1}. Therefore $|S| \le \frac1{\beta^k(G)-1}|U|$. The same argument applies to $T$, so we obtain $n = |S|+|T|+|U| \le \left(1 + \frac2{\beta^k(G)-1}\right)|U|$. Solving for~$|U|$ yields $|U| \ge \frac{\beta^k(G) - 1}{\beta^k(G) + 1} n$. Combining the bounds from the two subcases, $|U| \ge \min\left\{\frac{\beta^k(G) - 1}{\beta^k(G) + 1} n, \delta(G)\right\}$. Also, by Lemma~\ref{lem:minimum-degree-with-k}$(2)$, we have $\delta(G) \ge n-\frac{n-1}{\beta^k(G)} > \frac{\beta^k(G)-1}{\beta^k(G)} n > \frac{\beta^k(G)-1}{\beta^k(G)+1} n$. Therefore $\kappa(G) = |U| \ge \frac{\beta^k(G)-1}{\beta^k(G)+1}n$. 
\end{proof}

\section{Disjoint Perfect and Near-perfect Matchings}\label{sec:1-factors}
\subsection{The Non-bipartite Case}\label{subsec:1-factors-nonbip}

Based on the parity of the order of the considered graphs, we prove two lemmas which serve as the basis for proving Theorem~\ref{thm:1-factors-nonbip}. We begin by recalling Tutte's theorem, a well-known result on the existence of a perfect matching. 

\begin{theorem}[Theorem A in~\cite{Tutte1954ASP}]\label{thm:Tutte-1}
    A graph $G$ contains a perfect matching if and only if $q_G(U) \le |U|$ for every subset $U \subseteq V(G)$, where $q_G(U)$ denotes the number of components of $G-U$ with odd order.   
\end{theorem}

 When a graph has even order, we apply Tutte's theorem directly to guarantee a perfect matching.
\begin{lemma}\label{lem:1-factor-exist}
    Let $k \ge 2$ be an integer, and let $G$ be a graph of even order $n$. If $\beta^k(G) \ge 1$, then $G$ admits a perfect matching.
\end{lemma}
\begin{proof}
    To prove this, it suffices to consider the case $k = 2$ by Observation~\ref{obs:k_to_k-1}. Let $U$ be any subset of $V(G)$. We claim that $q_G(U) \le |U|$. Suppose that $U = \emptyset$. By Corollary~\ref{cor:connected}, $G$ is connected, and since $n$ is even, we have $q_G(U) = 0 = |U|$, as desired. Now suppose that $|U| \ge 1$, and let $\alpha$ be the number of components of $G-U$. By definition, $q_G(U) \le \alpha$. If $\alpha \le 1$, then clearly $q_G(U) \le 1 \le |U|$. If $\alpha \ge 2$, let $S$ be a set obtained by selecting one vertex from each component of $G-U$, so $|S| = \alpha \ge 2$. By the choice of $S$, each vertex not in $U$ can be adjacent to at most one vertex in $S$. Thus $\Lambda^2_G(S) \subseteq U$, and hence $|U| \ge |\Lambda^2_G(S)|\ge \beta^2(G)|S| \ge \alpha \ge q_G(U)$. In all cases, we have $q_G(U) \le |U|$. By Theorem~\ref{thm:Tutte-1}, $G$ has a perfect matching.
\end{proof}

We next consider graphs of odd order. Instead of directly showing that a near-perfect matching exists, we focus on hypomatchability, a strictly stronger property. A graph $G$ is said to be {\em hypomatchable} if, for every vertex $v \in V(G)$, the graph $G-v$ has a perfect matching. Each such matching is a {\em near-perfect matching} in $G$. Obviously, any hypomatchable graph must have odd order.
To show hypomatchability, we first prove the following corollary of Tutte's theorem.

\begin{corollary}\label{cor:Tutte_hypomatchable}
A graph $G$ with odd order $n$ is hypomatchable if and only if $q_G(U) \le |U|-1$ for every nonempty set $U \subseteq V(G)$.
\end{corollary}
\begin{proof}
The case $n = 1$ is trivial, so we may assume $n \ge 3$.
Suppose that $G$ is hypomatchable.
Let $U$ be any nonempty subset of $V(G)$, and let $v \in U$. Denote $U' = U \setminus \{v\}$ and $G' = G - v$. Clearly, $U' \subseteq V(G')$. Since $G$ is hypomatchable, $G'$ contains a perfect matching. Hence $G'$ satisfies the condition in Theorem~\ref{thm:Tutte-1}, and in particular, $q_{G'}(U') \le |U'| = |U| - 1$. Observe that removing $U'$ from $G'$ is equivalent to removing $U$ from $G$, so $G'-U' = G-U$ and $q_{G'}(U') = q_G(U)$. Thus we have $q_G(U) \le |U| - 1$. 

Suppose that $G$ satisfies $q_G(U) \le |U|-1$ for every nonempty set $U \subseteq V(G)$. Let $v$ be any vertex in $V(G)$. We show that $G' = G-v$ satisfies the condition in Theorem~\ref{thm:Tutte-1}, which implies that $G'$ contains a perfect matching. Let $U'$ be any subset of $V(G')$. By definition, $v \notin U'$. Denote $U = U' \cup \{v\}$. Then $\emptyset \ne U \subseteq V(G)$. As before, $q_{G'}(U') = q_G(U)$, and so $q_{G'}(U') \le |U|-  1=|U'|$ by assumption. Since $v$ was arbitrary, $G$ is hypomatchable.
\end{proof}

We apply this result in the following lemma for the odd-order case.
\begin{lemma}\label{lem:hypomatchable}
    Let $k \ge 2$ be an integer, and let $G$ be a graph of odd order $n$. If $\beta^k(G) \ge 1$, then $G$ is hypomatchable.
\end{lemma}
\begin{proof}
    As in Lemma~\ref{lem:1-factor-exist}, we consider the case $k = 2$. For the sake of contradiction, suppose that the condition in Corollary~\ref{cor:Tutte_hypomatchable} is violated. Then there exists a nonempty set $U \subseteq V(G)$ such that $q_G(U) \ge |U|$. Denote the odd-order components of $G-U$ by $C_1, \dots, C_{q_G(U)}$, and the even-order components (if any) by $D_1, \dots, D_m$. We claim that $q_G(U) \ne |U|$. Indeed, if $q_G(U) = |U|$, then 
    $|U|$ and the total number of vertices in the odd-order components have the same parity, contradicting $n$ odd. Thus $q_G(U) > |U|$. Since $U$ is nonempty, $q_G(U) \ge |U| + 1 \ge 2$. Let $S$ be the set consisting of one vertex from each odd-order component in $G-U$. Then $\Lambda^2(S) \subseteq U$, which gives $\frac{|\Lambda^2(S)|}{|S|} \le \frac{|U|}{|S|} = \frac{|U|}{q_G(U)} < 1$, contradicting $\beta^2(G) \ge 1$. Therefore $G$ satisfies the condition in Corollary~\ref{cor:Tutte_hypomatchable}, and $G$ is hypomatchable.
\end{proof}

We now introduce a technical lemma that will be repeatedly used in the subsequent sections. The proof uses the following notion of domination. Let $G$ be any graph, and let $F$ be a spanning subgraph of $G$. For a vertex $v \in V(G)$ and a subset $W \subseteq V(G)$, we say that $v$ {\em $F$-dominates} $W$ if $F$ contains all edges between $v$ and $W$; we say that $v$ {\em nearly $F$-dominates} $W$ if $F$ contains at least $|W|-1$ of these edges.


\begin{lemma}\label{lem:f-dominating}
Let $k\ge 2$ be an integer, and let $G$ be a graph with a spanning subgraph $F$. If $\beta^k(G) \ge 1$ and the maximum degree $\Delta(F) \le k-2$, then $\beta^2(G - E(F)) \ge 1$.
\end{lemma}
\begin{proof}
Let $H = G -E(F)$. Since $\beta^k(G) \ge 1$, it follows from Observation~\ref{obs-1} that $n \ge 2k$.
When $k = 2$, we have $\Delta(F) = 0$, so $F$ has no edges and hence $H = G$. In this case, $\beta^2(H) \ge 1$ holds by assumption. Therefore we may assume throughout the remainder of the proof that $k \ge 3$. 

Assume, to the contrary, that $\beta^2(H) < 1$. Then there exists some set $S \subseteq V(H)$ with $|S|\ge 2$ and $|\Lambda^2_H(S)| < |S|$. We first claim that $|S| < k$. Otherwise, either $|\Lambda^k_G(S)| \ge \beta^k(G)|S|$ or $|\Lambda^k_G(S)| = |V(G)|$, both of which imply that $|\Lambda^k_G(S)| \ge |S|$. For each vertex $v\in \Lambda^k_G(S)$, since $v$ is adjacent in $G$ to at least $k$ vertices of $S$ and $v$ has at most $k-2$ neighbors in $F$, it follows that $v$ is adjacent in $H$ to at least two vertices in $S$. Hence $|\Lambda^2_H(S)| \ge |\Lambda^k_G(S)| \ge |S|$, giving a contradiction.

Next, we show that for all nonempty subsets $T \subseteq S$, there are at most
\[
    (k-2) - (|T|-1)(k-|S|+1)
\]
vertices of $G$ that $F$-dominate $T$. To prove it, we induct on $|T|$. When $|T| =1$, let $T=\{v\}$. Since $d_F(v) \le \Delta(F) \le k-2$, there are at most $k-2$ vertices that $F$-dominate $\{v\}$. So the base case holds.

For some $t \ge 1$, assume that the bound holds for all sets $T \subseteq S$ with $|T|=t$. Let $T \subseteq S$ with $|T| = t + 1$ be arbitrary. Since $|S| < k$, we have $|T| < k$. 

Let $U$ be a set of $k$ vertices containing $T$ and $k-t-1$ additional vertices, chosen so that as many of them as possible nearly $F$-dominate $T$. Then $|\Lambda^k_G(U)| \ge |U| = k$. Every vertex $x \in \Lambda^k_G(U)$ is adjacent in $G$ to all of $U$, meaning that $x$ itself is not in $U$ and that $x$ is adjacent in $G$ to all of $T$.
This would place $x$ in $\Lambda^2_H(T) \subseteq \Lambda^2_H(S)$, unless at least $|T|-1 = t$ of the edges from $x$ to $T$ are in $F$; that is, unless $x$ nearly $F$-dominates $T$.

Since $|\Lambda^2_H(S)| < |S|$, at least $k-|S|+1$ of the vertices in $\Lambda^k_G(U)$ must nearly $F$-dominate $T$. By our choice of $U$, all $k-t-1$ vertices of $U \setminus T$ must also nearly $F$-dominate $T$, for a total of $2k - |S| - t$ such vertices.

Each vertex that nearly $F$-dominates $T$ falls under one of two cases: it either
\begin{enumerate}[label=(\arabic*)]
    \item $F$-dominates exactly one $t$-element subset of $T$, or
    \item $F$-dominates all $t$-element subsets of $T$, as well as $T$ itself.
\end{enumerate}
By assumption, at most $(k-2) - (t-1)(k-|S|+1)$ vertices $F$-dominate each $t$-element subset of $T$. If we sum over all $t+1$ of these subsets, then every vertex in case (1) is counted once, and every vertex in case (2) is counted $t+1$ times. Let $d$ denote the number of vertices of $G$ that $F$-dominate $T$. Then at most
\[
    (t+1)\Big((k-2) - (t-1)(k-|S|+1)\Big) - td 
\]
vertices nearly $F$-dominate $T$. We obtain the inequality
\[
    (t+1)\Big((k-2) - (t-1)(k-|S|+1)\Big) - td \ge 2k - |S| - t.
\]
Solving for an upper bound on $d$, we get
\begin{align*}
d   &\le \frac{(t+1)(k-2) - (t^2-1)(k-|S|+1) - 2k + |S| + t}{t} \\
    &= \frac{t(k-2) - t^2(k-|S|+1)}{t} + \frac{(k-2) + (k-|S|+1) - 2k + |S| + t}{t} \\
    &= (k-2) - t(k-|S|+1) + \frac{t-1}{t}.
\end{align*}
Since both $d$ and $(k-2) - t(k-|S|+1)$ are integers, while $\frac{t-1}{t} < 1$, it follows that 
\[
    d \le (k-2) - t(k-|S|+1), 
\]completing the induction step.

Setting $T = S$, the upper bound simplifies to $(k-2) - (|S|-1)(k-|S|+1) = (|S|-k)(|S|-2)-1$, which is negative when $2 \le |S| < k$. Since the number of vertices in $G$ that $F$-dominate $S$ cannot be negative, this yields a contradiction. Therefore no such set $S$ can exist, and we conclude that $\beta^2(H) \ge 1$.
\end{proof}

We finish this section with a short proof of Theorem~\ref{thm:1-factors-nonbip}.

\begin{proof}[Proof of Theorem~\ref{thm:1-factors-nonbip}]
We proceed by induction on $k$. When $k = 1$, there is nothing to prove, and when $k=2$, the theorem follows by Lemma~\ref{lem:1-factor-exist} if $n$ is even, and by Lemma~\ref{lem:hypomatchable} if $n$ is odd.

Now assume for some $k > 2$ that the theorem holds for $k-1$. Let $G$ be a graph of order $n$ with $\beta^k(G) \ge 1$. Then by Observation~\ref{obs:k_to_k-1}, we have $\beta^{k-1}(G) \ge 1$. By the inductive hypothesis, $G$ contains $k-2$ pairwise disjoint perfect matchings (if $n$ is even) or near-perfect matchings (if $n$ is odd).

Let $M$ be the union of these matchings, and let $H = G-E(M)$. Clearly, $\Delta(M) \le k - 2$ and $|V(H)| = |V(G)| = n$. By Lemma~\ref{lem:f-dominating} with $F=M$, we have $\beta^2(H) \ge 1$. Then, applying Lemma~\ref{lem:1-factor-exist} or~\ref{lem:hypomatchable} to $H$, depending on the parity of $n$, we obtain a perfect or near-perfect matching in $H$ that is disjoint from the $k-2$ matchings we found in $G$. In total, $G$ contains $k-1$ pairwise disjoint perfect or near-perfect matchings, completing the induction step. We conclude that the theorem holds for all positive integers $k$.
\end{proof}

\subsection{The Bipartite Case}\label{subsec:1-factors-bip} 
In this section, we begin with two classical results on matchings in bipartite graphs. Let $G$ be a bipartite graph with bipartition $(X, Y)$.

The first is Hall's Theorem, which characterizes when $G$ contains a matching that covers all of $X$.

\begin{theorem}[Theorem 2.1.2 in~\cite{diestel}]\label{thm:Hall's}
    A bipartite graph $G$ with bipartition $(X, Y)$ contains a matching covering all vertices of $X$ if and only if $|\Lambda(S)| \ge |S|$ for all subsets $S \subseteq X$.
\end{theorem}

The second is a generalization due to Lebensold~\cite{LEBENSOLD1977}, which provides a necessary and sufficient condition for the existence of $k$ disjoint matchings, each covering all of $X$. For a subset $S \subseteq X$, define
\[
    L^k(S) = \sum_{y \in Y} \min\{k, |\Lambda(y) \cap S|\}.
\]

\begin{theorem}[\cite{LEBENSOLD1977}]\label{thm:Lebensold}
    Let $G$ be a bipartite graph with bipartition $(X, Y)$. Then $G$ contains $k$ disjoint matchings covering $X$ if and only if $L^k(S) \ge k|S|$ for all $S \subseteq X$.
\end{theorem}

\begin{proof}[Proof of Theorem~\ref{thm:1-factor-bip}]
    As $\beta^k(G, X) \ge 1$, it follows from the definition that $|X| \ge k$.
    When $k = 1$, $G$ clearly satisfies $|\Lambda(S)| \ge |S|$ for any set $S \subseteq X$, and Theorem~\ref{thm:Hall's} guarantees the existence of a matching covering $X$ in $G$. So we may assume $k \ge 2$. To apply Theorem~\ref{thm:Lebensold}, it suffices to verify that $L^k(S) \ge k|S|$ for all $S \subseteq X$. Let $S$ be any subset of $X$. 
    
    
    Suppose that $|S| \ge k$. Since $|\Lambda^k(S)| \ge \beta^k(G, X)|S| \ge |S|$, there are at least $|S|$ vertices in $Y$ that are adjacent to at least $k$ vertices in $S$, which gives $L^k(S) \ge k|S|$. 
    
    Now suppose that $|S| \le k-1$. We claim that each vertex in $X$ has a degree of at least $k$. Indeed, for any vertex $x \in X$, there is a $k$-set $S_x \subseteq X$ such that $x \in S_x$. Since $|S_x| = k$, then every vertex in $\Lambda^k(S_x)$ must be a neighbor of $x$. Hence $|\Lambda^k(S_x)| \ge \beta^k(G, X)|S_x| \ge k$ implies $d(x) \ge k$. Observe that if $|S| \le k-1$, then for every $y \in Y$, we have $|\Lambda(y) \cap S| \le |S| < k$. It follows that $L^k(S) = \sum_{y \in Y}|\Lambda(y) \cap S|$ counts precisely the number of edges between $S$ and $Y$. By double counting, we obtain $L^k(S) = \sum_{x \in S} d(x) \ge k|S|$, as required. 
    
    Therefore, in either case, $L^k(S) \ge k|S|$ holds, and by Theorem~\ref{thm:Lebensold}, $G$ contains $k$ disjoint matchings, each covering all of $X$.
\end{proof}

\section{The First Proof of Theorem~\ref{thm:k-factor}}\label{sec:first-proof}

We begin by recalling the result from Bar\'at \textit{et al.}~\cite{Hungarian}, as mentioned in Section~\ref{sec:introduction}. 

\begin{theorem}[Theorem 3.4 in~\cite{Hungarian}]\label{thm:2-factor}
    Let $G$ be a graph on at least two vertices. If  $\beta^2(G) \ge 1$,  then $G$ contains a $2$-factor.
\end{theorem}

\begin{proof}[Proof of Theorem~\ref{thm:k-factor}]
The case $k=2$ follows from Theorem~\ref{thm:2-factor}.

Suppose $k \ge 3$, and that Theorem~\ref{thm:k-factor} holds for any integer $i$ with $2\le i \le k-1$.  Let $G$ be a graph on $n$ vertices such that $nk$ is even and $\beta^k(G) \ge 1$. 

If $n$ is even, then by Theorem~\ref{thm:1-factors-nonbip}, $G$ contains $k-1$ pairwise disjoint perfect matchings. Let $F$ be the subgraph spanned by the union of $k-2$ of these matchings. Then $F$ is a $(k-2)$-factor of $G$.
If $n$ is odd, then $k$ is even as $nk$ is even, and so $k \ge 4$ and $n(k-2)$ is even. By Observation~\ref{obs:k_to_k-1},   $\beta^{k-2}(G) \ge 1$. 
By induction, $G$ contains a $(k-2)$-factor, also denoted by $F$. Thus, in either case, $G$ contains a $(k-2)$-factor $F$. 

By Lemma~\ref{lem:f-dominating}, $\beta^2(G-E(F)) \ge 1$, so Theorem~\ref{thm:2-factor} implies that $G-E(F)$ contains a $2$-factor $M$. Then $M\cup F$ is a $k$-factor of $G$. 

It remains to describe a family of graphs that demonstrates that the bound $1$ in the condition $\beta^k(G) \ge 1$ is tight. For any integer $k \ge 2$, let $n \ge 2k + 2$ be even, and consider the split graph
\[
G := \overline{K}_{\frac{n}2 +1} + K_{\frac{n}{2}-1}. 
\]
See Figure~\ref{fig:split_K3_5} for an explicit example.
\begin{figure}[H]
    \centering
    \begin{tikzpicture}[scale=0.7, every node/.style={fill=black, circle, inner sep=1.5pt}]
    
    \node[label=above:$v_1$] (v1) at (-2, 1) {};
    \node[label=below:$v_2$] (v2) at (-2, -1) {};
    \node[label=left:$v_3$] (v3) at (-3, 0) {};

    \node[label=right:$u_1$] (u1) at (1, 2) {};
    \node[label=right:$u_2$] (u2) at (1, 1) {};
    \node[label=right:$u_3$] (u3) at (1, 0) {};
    \node[label=right:$u_4$] (u4) at (1, -1) {};
    \node[label=right:$u_5$] (u5) at (1, -2) {};

    \foreach \a/\b in {v1/v2, v1/v3, v2/v3} {
        \draw (\a)--(\b);
    }

    \foreach \u in {u1, u2, u3, u4, u5} {
        \foreach \v in {v1, v2, v3} {
            \draw (\u)--(\v);
        }
    }
    \end{tikzpicture}
    \caption{Example for $n = 8$.}
    \label{fig:split_K3_5}
\end{figure}
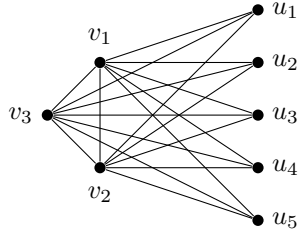

Such a graph $G$ does not contain a $k$-factor. Otherwise, $k (\frac n2 +1) = |E(\overline{K}_{\frac n2 +1}, K_{\frac n2 -1})| \le k (\frac{n}2 -1)$, a contradiction.

We claim that $\beta^k(G)$ tends to $1$ from below as $n$ becomes sufficiently large relative to $k$, while all other conditions of the theorem remain satisfied, showing that the bound on $\beta^k(G)$ cannot be replaced with any value smaller than 1 in the theorem.
Clearly, the graph $G$ has order $n \ge 2k+2$, and $nk$ is even since $n$ is even. Let $S$ be any subset of $V(G)$ such that $|S| \ge k$ and $\Lambda^k(S) \ne V(G)$. Define $s_I = |S \cap V(\overline{K}_{\frac n2 +1})|$ and $s_C = |S \cap V(K_{\frac n2 -1})|$, so $|S|=s_I + s_C$. Suppose that $|S| \ge k+1$. Each vertex in $K_{\frac{n}{2}-1}$ is adjacent to all other vertices in $G$, so $V(K_{\frac{n}{2}-1}) \subseteq \Lambda^k(S)$. Since $\Lambda^k(S) \ne V(G)$, we must have $s_C \le k-1$, and thus $\Lambda^k(S) = V(K_{\frac{n}{2}-1})$. To minimize the ratio $\frac{|\Lambda^k(S)|}{|S|}$, we maximize $S$, which is achieved when $s_C = k-1$ and $s_I = \frac{n}{2}+1$. This gives us 
\[
\frac{|\Lambda^k(S)|}{|S|} = \frac{n-2}{n+2k} < 1.
\]
Now suppose that $|S| = k$. In this case, $\Lambda^k(S) \cap S = \emptyset$. If $s_I \ge 1$, then $s_C \le k-1$. This implies that $\Lambda^k(S) = V(K_{\frac{n}{2}-1}) \setminus S$, and thus
\[
\frac{|\Lambda^k(S)|}{|S|} = \frac{\frac{n}{2} - 1 -s_C}{k} \ge \frac{n-2k}{2k},
\]
with equality when $s_I = 1$ and $s_C = k-1$. If $s_I = 0$, then $s_C = k$ and $\Lambda^k(S) = V(G) \setminus S$, yielding
\[
\frac{|\Lambda^k(S)|}{|S|} = \frac{n-k}{k} > 1.
\]
By definition, $\beta^k(G) < 1$. Moreover, when $n \gg k$, we have $\beta^k(G) = \frac{n-2}{n+2k}$, which is strictly less than $1$ but approaches $1$.
\end{proof}

\section{Further Preliminaries on Tutte's \texorpdfstring{$k$}{k}-Factor Theorem}\label{sec:Pre}
Let $G$ be a graph. Given any subsets $S$ and $T$ of $V(G)$, we write $e_G(S, T)$ for the number of edges with one endpoint in $S$ and the other in $T$, counting each edge exactly once even if $S \cap T \ne \emptyset$. If $S = T$, then $e_G(T):= e_G(T, T)$ denotes the number of edges with both ends in $T$; and if $S = \{s\}$, we write $e_G(s, T)$ for $e_G(\{s\}, T)$.
Let $k$ be a positive integer, and let $(S, T, U)$ be an ordered triple such that $S, T, U$ form a partition of $V(G)$. We refer to such a triple as an \emph{ordered partition} of $V(G)$. A component $C$ of $G[U]$ is said to be {\it odd} (resp. {\it even}) with respect to $(S, T, U)$ and $k$ if $k|C| + e(C, T)$ is odd (resp. even). We write $q(S, T, U)$ for the number of odd components of $G[U]$ (with respect to $(S, T, U)$ and $k$). Our method is based on Tutte's $k$-factor theorem stated below.

\begin{theorem}[Theorem C in~\cite{Tutte1954ASP}]\label{thm:Tutte}
Let $G$ be a graph and $k$ be an integer satisfying $0 \le k \le \delta(G)$. The graph $G$ contains a $k$-factor if and only if, for any ordered partition $(S, T, U)$ of $V(G)$, the value
\begin{align*}
\delta_G(S, T,U) := k|S| - k|T| +\sum_{v \in T}d_{G-S}(v) - q(S, T, U)
\end{align*}
is nonnegative.
\end{theorem}

We call an ordered partition $(S,T,U)$ a \emph{barrier} if $\delta_G(S,T,U) < 0$. In particular, we call $(S,T,U)$ a \emph{maxmin barrier} if, of all barriers, it
\begin{enumerate}[label=(\arabic*)]
    \item maximizes $|S|$; and
    \item subject to (1), minimizes $|T|$.
\end{enumerate}

The notion of a maxmin barrier, along with several of its properties, was previously investigated in the context of Tutte's 2-factor theorem~\cite{ GRIMM2023, kannoshan2019}. Here, we extend the study to general $k \ge 2$, including the $k = 2$ case and refining some aspects of the earlier results.

Let $k \ge 2$ be an integer, and let $G$ be a graph on $n$ vertices such that $nk$ is even and $G$ has no $k$-factor. 
\begin{lemma}\label{lem:barrier-range}
For all ordered partitions $(S, T, U)$, $\delta_G(S, T, U) \equiv 0 \pmod 2$. In particular, if $(S, T, U)$ is a barrier, then $\delta_G(S, T, U) \le -2$.
\end{lemma}
\begin{proof}
    Let $(S, T, U)$ be any ordered partition in $G$, and let $C_1,...,C_m$ denote all components of $G[U]$. Then $q(S,T,U)$ can be thought of as the sum of $(k|C_i| + e(C_i,T)) \bmod 2$ over all $i$. Therefore $q(S,T,U) \equiv k|U| + e(U, T) \pmod 2$. Also, $\sum_{v\in T}d_{G-S}(v) = \sum_{v \in T} d_T(v) + \sum_{v \in T} d_U(v) = 2e(T)+e(U, T)$.
    We consider two cases based on the parity of $k$.
    
    Suppose that $k$ is even. Then $k|S| - k|T| \equiv 0 \pmod 2$, so
    \[
    \begin{aligned}
    \delta_G(S, T, U) &\equiv \big(k|S| - k|T|\big) + \big(2e(T) + e(U, T)\big) - \big(k|U| + e(U, T)\big) \\
    &= \big(k|S| - k|T|\big) + 2e(T) - k|U| \\
    &\equiv 0 \pmod 2.
    \end{aligned}
    \]

    Suppose that $k$ is odd. Then $n = |V(G)|$ must be even, since $nk$ is even. We now claim that $\delta_G(S, T, U) \equiv |V(G)| \pmod 2 \equiv 0 \pmod 2$. In this case, $k|S| - k|T| \equiv |S| - |T| \equiv |S| + |T| = |V(G)| - |U| \equiv |V(G)| + |U| \pmod 2$. Altogether, 
    \[
    \begin{aligned}
    \delta_G(S,T,U) &\equiv \big(|V(G)| + |U|\big) + \big(2e(T) + e(U, T)\big) - \big(k|U| + e(U, T)\big) \\
    &= |V(G)| + (1-k)|U| + 2e(T) \\
    &\equiv |V(G)| \pmod 2. 
    \end{aligned}
    \]
    In both cases, $\delta_G(S,T,U)$ is even.  In particular, $\delta_G(S, T, U) \le -2$ whenever $(S, T, U)$ is a barrier.
\end{proof}

The next lemma describes several structural properties of a maxmin barrier in $G$, which will be crucial in the next two sections.

\begin{lemma}\label{lem:maxmin-barrier}
    Let $(S, T, U)$ be a maxmin barrier of $G$.  
   \begin{enumerate}[label=(\roman*)]
   \item For all vertices $w \in U$, $e(w, T) \le k-1$; moreover, for all vertices $w$ in an even component of $G[U]$, $e(w, T) \le k-2$. 
   \item The maximum degree $\Delta(G[T]) \le k-2$.
   \item For any odd component $C$ of $G[U]$ and any vertex $w \in T$, $e(w, T) + e(w, C) \le k-1$.
   \item For any vertex $w \in T$, $e(w, T)+\sum_{C}e(w, C) \le k-2$, where $C$ ranges over all even components of $G[U]$.
   \end{enumerate}
\end{lemma}

\begin{proof}
For an arbitrary vertex $w \in U$, let $(S', T, U') = (S \cup \{w\}, T, U \setminus \{w\})$. Then
\begin{align*}
\delta_G(S', T, U') 
    &= k(|S|+1) - k|T| + \sum_{v \in T} d_{G-S}(v) - e(w, T) - q(S', T, U') \\
    &= \delta_G(S,T,U) + k - e(w, T) + \Big(q(S,T,U) - q(S',T,U')\Big).
\end{align*}
The difference $q(S,T,U) - q(S',T,U')$ is at most $1$, since only the component containing $w$ can change parity. Also, by Lemma~\ref{lem:barrier-range}, $\delta_G(S,T,U) \le -2$. Thus $\delta_G(S', T, U') \le (k - 1) - e(w, T)$. By the maxmin nature of $(S,T,U)$, we know $\delta_G(S', T, U') \geq 0$.  Hence $e(w,T) \leq k-1$, with equality only if $q(S,T,U) - q(S',T,U')=1$.  This difference is equal to 1 if and only if moving $w$ from $U$ to $S$ decreases the number of odd components by 1, which happens only if $w$ is in an odd component in $G[U]$; otherwise, $e(w,T) \leq k-2$. This establishes $(i)$.

For an arbitrary vertex $w \in T$, let $(S, T', U') = (S, T \setminus \{w\}, U \cup \{w\})$. Suppose that $w$ has neighbors in $m$ distinct odd components of $G[U]$; clearly, $m \le e(w, U)$. Since $|T|$ is minimal among all barriers with $S$ as the first part, it follows that $\delta_G(S, T', U') \ge 0$. Also, by the definition of an odd component, if $C$ is an odd component of $G[U]$ and $e(w, C) = 0$, then $C$ remains odd in $G[U']$. Hence $q(S, T', U') \ge q(S, T, U) - m$. Then
\begin{align*}
0 \le \delta_G(S, T', U') &=
    k|S| - k(|T|-1) + \sum_{v \in T} d_{G-S}(v) - d_{G-S}(w) - q(S, T', U') \\
    &= \delta_G(S,T,U) + k - d_{G-S}(w) + \Big(q(S,T,U) - q(S,T',U')\Big) \\
    &\le \delta_G(S,T,U) + k - d_{G-S}(w) + m.
\end{align*}
By Lemma~\ref{lem:barrier-range}, we have $0 \le \delta_G(S, T', U') \le  k - 2 + m - d_{G-S}(w)$, or $d_{G-S}(w) \le k-2+m$. On the other hand, $d_{G-S}(w) = e(w, T') + e(w, U) \ge e(w, T') + m$. Together, this implies $e(w, T') \le k-2$. 
Furthermore, for each odd component $C$, we have $e(w, T') + e(w, C) \le k-2+m-(m-1) = k-1$, since $w$ must send at least $m-1$ edges to the other odd components in which it has a neighbor, regardless of whether $w$ has a neighbor in $C$ or not. Also, since at least $m$ of the edges from $w$ in $G-S$ go to odd components, we have $e(w, T') + \sum_{C}e(w, C) \le k-2$, where the sum is over all even components $C$ of $G[U]$. This establishes $(ii)$, $(iii)$, and $(iv)$.
\end{proof}

\section{An Alternate Proof of Theorem~\ref{thm:k-factor}}\label{sec:k-factor}

\begin{proof}
    Let $k \ge 2$ be an integer, and let $G$ be a graph on $n$ vertices such that $\beta^k(G) \ge 1$ and $nk$ is even. Suppose that, on the contrary, $G$ does not admit a $k$-factor. Then $G$ admits a barrier. Let $(S, T, U)$ be a maxmin barrier. Denote by $C_1,\dots, C_{m_1}$ the odd components of $G[U]$, and by $D_1,\dots, D_{m_2}$ the even components. We partition the odd components of $G[U]$ into two classes: those with no neighbors in $T$ and those with at least one neighbor in $T$. We denote these classes by $\mathcal{H}_0$ and $\mathcal{H}_1$, respectively. Let $h_0 = |\mathcal{H}_0|$ and $h_1 = |\mathcal{H}_1|$. Then $q(S, T, U) = m_1 = h_0 + h_1$ and $e(U, T) = \sum_{i \in [m_1]}e(C_i, T) + \sum_{i \in [m_2]}e(D_i, T) \ge h_1$.  Based on the parity of $k$, we have two cases. 

    \textbf{Case 1: $k$ is even.}
    We first claim that $|S| < |T|$. When $k$ is even, each component $C$ of $G[U]$ is odd if and only if $e(C, T)$ is odd. In particular, this implies that $e(C_i, T) \ge 1$ for each $i \in [m_1]$, and thus $\sum_{v \in T}d_{G-S}(v) \ge e(U, T) \ge m_1 = q(S, T, U)$. It follows that $0 > \delta_G(S, T, U) \ge  k|S| - k|T|$, which yields $|T| > |S|$. 
    
    If $|T| \ge k$, then by Lemma~\ref{lem:maxmin-barrier}$(i)-(ii)$, $\Lambda^k(T) \subseteq S$, and thus $|S| \ge |\Lambda^k(T)| \ge \beta^k(G)|T| \ge |T|$, a contradiction. So $|T| \le k-1$, and hence $|S| + |T| \le 2|T| - 1 \le 2k-3$. As $\beta^k(G) \ge 1$, it follows that $n \ge 2k$, so $G[U]$ contains at least one component. We claim that for each component $C$ of $G[U]$, whether even or odd, we have $|C| \le k-|T|-1$. Otherwise, let $C$ be a component of $G[U]$ such that $|C| \ge k-|T|$. Let $Q:= T\cup C$, so $|Q| \ge k$. 
    Note that for each vertex $v \in U\setminus V(C)$, $|\Lambda(v) \cap Q| = |\Lambda(v) \cap T| \le k-1$, since $\Lambda(v) \cap V(C) = \emptyset$. So $(U\setminus V(C)) \cap \Lambda^k(Q) = \emptyset$. In addition, by Lemma~\ref{lem:maxmin-barrier}$(iii)-(iv)$, every vertex $v \in T$ has at most $k-1$ neighbors in $Q$. So $T \cap \Lambda^k(Q) = \emptyset$. Therefore $\Lambda^k(Q) \subseteq V(C) \cup S$, implying $\frac{|\Lambda^k(Q)|}{|Q|} \le \frac{|V(C)|+|S|}{|V(C)|+|T|} < 1$, which gives a contradiction to $\beta^k(G) \ge 1$. Thus every component of $U$ has at most $k-|T|-1$ vertices. 
     
    Now we consider the set $U' = U \cup T$. Note $|U'| = n-|S| > 2k-(k-1)=k+1$. We claim that $\Lambda^k(U') \subseteq S \cup T$. Let $v \in U$, and suppose that $v$ lies in a component $C$ of $G[U]$. We know $v$ has no neighbors in other components of $G[U]$, at most $|C|-1$ neighbors within $C$, and at most $|T|$ neighbors in $T$. Together, this implies $|e(v, U')| \leq |C|-1 + |T| < k$, and $v \notin \Lambda^k(U')$. Thus $\Lambda^k(U') \subseteq S \cup T$.  
     We now claim that $|U| > |S|$. It suffices to show $q(S, T, U) > |S|$. Since $\delta_G(S,T,U)<0$ and $\delta(G) \ge (\beta^k(G)+1)k -1 \ge 2k-1$ by Lemma~\ref{lem:minimum-degree-with-k}, we have
    \begin{align*}
    0&> k|S|-k|T|+\sum_{v\in T}d_{G-S}(v)-q(S,T, U) \\
    &\geq k|S|-k|T|+|T|(2k-1-|S|)-q(S,T,U)\\
    &= (k-|T|)|S|+(k-1)|T|-q(S,T,U)\\
    &\geq |S|+0-q(S,T,U)
    \end{align*}
    and therefore $q(S,T,U)>|S|$ as needed.
    Then $\frac{|\Lambda^k(U')|}{|U'|} \le \frac{|S|+|T|}{|U|+|T|} < 1$, contradicting $\beta^k(G) \ge 1$.

    \textbf{Case 2: $k$ is odd.} Then $k \ge 3$, and by Lemma~\ref{lem:barrier-range}, we have 
    \[
    \begin{aligned}
    -2 \ge \delta_G(S,T,U) &= k|S| - k|T| + \sum_{v \in T}d_{G-S}(v) - q(S, T, U)   \\
    &\ge k|S| - k|T| + e(U, T) - (h_0 + h_1) \\
    &\ge k|S| - k|T| - h_0.
    \end{aligned}
    \] 
    This implies that $k|S| \le k|T|+ h_0 - 2$, or $|S| \le |T| + \frac{1}{k}(h_0 - 2) < |T| + h_0$. We claim that $T \ne \emptyset$. If $T = \emptyset$, then $|S| \le \frac{1}{k}(h_0 - 2)$. Clearly, $S \ne \emptyset$; otherwise, $S=T=\emptyset$ and $G[U] = G$ is connected by Corollary~\ref{cor:connected}. Since $nk$ is even, $G[U]$ itself is an even component, and hence $\delta_G(S, T, U) =0$, a contradiction. So $1 \le |S| \le \frac1k (h_0 - 2)$, and thus $h_0 \ge k + 2$. Let $Q$ be the set consisting of a vertex from each component in $\mathcal{H}_0$. 
    Then $|Q| = h_0 \ge k+2$ and every vertex $v \notin S$ can be adjacent to at most one vertex in $Q$, which implies $\Lambda^k(Q) \subseteq S$. Thus $|S| \ge |\Lambda^k(Q)| \ge \beta^k(G)|Q| \ge h_0 > \frac 1k h_0$, contradicting $|S| \le \frac1k (h_0-2)$. So $|T| \ge 1$. 
    
    We now state the following claim, which gives an upper bound on the size of $\bigcup_{C \in \mathcal{H}_0}V(C)$. 

    \begin{claim}\label{claim:zero-vertex-size}
    Let $U_0 := \bigcup_{C \in \mathcal{H}_0} V(C)$.  
    \begin{itemize}
    \item If $|T| \le k - 1$, then $|U_0| \le k - |T| - 1$.
    \item If $|T| \ge k$, then $U_0 = \emptyset$, i.e. $h_0 = 0$.
    \end{itemize}
    \end{claim}
    \begin{proof}
    To justify the first item, assume by way of contradiciton that $|T| \le k-1$ and $|U_0| \ge k -|T|$. Then $|U_0| \ge 1$, and hence $h_0 \ge 1$. There exists a set $T'$ with $|T'| \ge \max\{k, |T|+h_0\}$ that satisfies the following conditions:
    \begin{enumerate}[label=(\alph*)]
    \item $T \subseteq T'$, and
    \item $T'$ contains at least one and at most $k-1$ vertices from each component in $\mathcal{H}_0$.
    \end{enumerate}
    To prove the existence of such a set $T'$, we proceed as follows. If $h_0 \ge k-|T|$, then we select exactly one vertex from each component in $\mathcal{H}_0$ and combine these vertices with $T$ to form the set $T'$. Then $|T'| = |T| + h_0 \ge k$. If $h_0 \le k-|T|-1$, we first combine $T$ with exactly one vertex from each component in $\mathcal{H}_0$, and then we add $k-(|T|+h_0)$ additional vertices from $U_0$ to create $T'$.  These vertices of $U_0$ exist since $|U_0| \ge k - |T|$, and $|T|\geq 1$ and $h_0 \geq 1$ ensure that $k-(|T|+h_0) \leq k-2$, so no component in $\mathcal{H}_0$ contributes more than $k-1$ to $T'$. Therefore there is a set $T'$ satisfying $|T'| \ge \max\{k, |T|+h_0\}$ and conditions $(a)$ and $(b)$.
    
    We claim that $\Lambda^k(T') \subseteq S$. 
    Let $v \in U$; say $v \in V(C)$. If $C \in \mathcal{H}_0$, then $e(C, T) = 0$ and $|\Lambda(v) \cap T'| = |\Lambda(v) \cap (T'\setminus T)| \le k-1$ by $(b)$. If $C \notin \mathcal{H}_0$, then $\Lambda(v)\cap U_0 =\emptyset$, and thus $|\Lambda(v) \cap T'| = |\Lambda(v) \cap T| \le |T| \le k-1$. In either case, we have $v \notin \Lambda^k(T')$. So $\Lambda^k(T') \subseteq S \cup T$. Also, by the definition of $U_0$, we have $e(U_0, T) = 0$. This implies that for each $v \in T$, $|\Lambda(v) \cap T'| = |\Lambda(v) \cap T| \le k-1$, leading to $\Lambda^k(T') \cap T = \emptyset$. So $\Lambda^k(T') \subseteq S$, and thus $|S| \ge |\Lambda^k(T')| \ge |T'| \ge \max\{k, |T|+h_0\}$, contradicting our earlier observation that $|S| < |T|+h_0$. Therefore $|U_0| \le k-|T|-1$.

    To justify the second item, we assume $|T|\ge k$ and, to the contrary, that $U_0 \ne \emptyset$. Then $h_0 \ge 1$. As in the previous case, let $T'$ be the set obtained by adding exactly one vertex from each component in $\mathcal{H}_0$ to $T$; since $|T| \ge k$, we have $|T'| = |T| + h_0 \ge k + 1$. We claim as before that $\Lambda^k(T') \subseteq S$, which would again yield a contradiction. Let $v \in U$; say $v \in V(C)$. If $C \in \mathcal{H}_0$, then $|\Lambda(v) \cap T'| \le 1$. If $C \notin \mathcal{H}_0$, then $|\Lambda(v) \cap T'| = |\Lambda(v) \cap T| \le k-1$ by Lemma~\ref{lem:maxmin-barrier}$(i)$. Thus $\Lambda^k(T') \cap U = \emptyset$. Also, $e(U_0, T) = 0$. Combined with Lemma~\ref{lem:maxmin-barrier}$(ii)$, this implies that for each $v \in T$, $|\Lambda(v) \cap T'| = |\Lambda(v) \cap T| \le k - 2$. It follows that $\Lambda^k(T')\cap T = \emptyset$, and hence $\Lambda^k(T') \subseteq S$. 
    \end{proof}
    
    Suppose that $|T| \ge k$. By Claim~\ref{claim:zero-vertex-size}, we know $h_0 = 0$, and thus $|S| < |T|$ as $\delta_G(S, T, U) < 0$. By Lemma~\ref{lem:maxmin-barrier}$(i)-(ii)$, each vertex $v \notin S$ satisfies $|\Lambda(v) \cap T| \le k-1$. This implies $\Lambda^k(T) \subseteq S$, that is, $|S| \ge |\Lambda^k(T)|$. Then $\frac{|\Lambda^k(T)|}{|T|} \le \frac{|S|}{|T|} < 1$, contradicting $\beta^k(G) \ge 1$.

    Suppose that $1 \le |T| \le k-1$. 
    By Lemma~\ref{lem:minimum-degree-with-k} with $\beta^k(G) \ge 1$, we have $d_{G-S}(v) \ge 2k-1-|S|$ for each $v \in T$. Then, 
    \[
    \begin{aligned}
     -2 \ge \delta_G(S,T,U) &\ge k|S|+(2k-1-|S|)|T|-k|T|-q(S,T,U)\\
     &=(k-|T|)|S|+(k-1)|T|-q(S,T,U),
    \end{aligned}
    \]
    which gives us $q(S,T,U) \ge (k-|T|)|S|+(k-1)|T| + 2$. Since $k \ge 3$ and $1 \le |T| \le k-1$, both inequalities $q(S, T, U) \ge k+1$ and $q(S, T, U) \ge |S| + |T| + 2$ hold. Let $Q$ be a vertex set consisting of a vertex from each odd component of $G[U]$, so $|Q| = q(S, T, U) \ge k+1$. For every $v\in U$, by definition, we have $e(v, Q) \le 1 <k$, and so $\Lambda^k(Q) \subseteq S\cup T$. Hence $|S| + |T| \ge |Q| = q(S, T, U) \ge |S|+|T|+2$, which is a contradiction.     
    
    In both cases, we arrive at a contradiction. Therefore $G$ contains a $k$-factor, as claimed.
\end{proof}

\section{Proof of Theorem~\ref{thm:k+1-factor-split-graph}}\label{sec:split}

This section focuses on split graphs. The proof of Theorem~\ref{thm:k+1-factor-split-graph} is first given for the base case $k=2$ using Tutte's $3$-factor theorem, and then extended to all $k \ge 2$.

\begin{observation}\label{obs:Y-range-split}
    Let $k \ge 2$ be an integer and $G$ be a split graph with split partition $(X, Y)$ such that $X$ is an independent set and $Y$ is a clique. If $\beta^k(G) \ge 1$, then $|Y| \ge \max\{k, |X| + k-1\}$.
\end{observation}
\begin{proof}
    Since $\beta^k(G) \ge 1$, Observation~\ref{obs-1} implies that $|G| = |X| + |Y| \ge 2k$. We may assume $X \ne \emptyset$; otherwise, $G$ is a complete graph, and so the result holds. Let $Q \supseteq X$ such that $|Q\cap Y| = \min\{k-1, |Y|\}$. Since $|X|+|Y| \ge 2k$, it follows that $|Q| \ge k$. Then, for each vertex $x \in X$, $|\Lambda(x) \cap Q| \le |Q \cap Y| \le k-1$, as $X$ is independent. Therefore $\Lambda^k(Q) \subseteq Y$, and we conclude $|Y| \ge \beta^k(G)|Q| \ge |X|$. In particular, we have $|Y| \ge k$, which in turn shows that $|Y| \ge |Q| = |X| + k - 1$.
\end{proof}

\subsection{The Base Case}
Let $k = 2$, and let $G$ be a graph of even order $n$ with $\beta^2(G) \ge 1$. Then, $n \ge 4$. Assume that $G$ contains no $3$-factor. By Tutte's $3$-factor Theorem, $G$ admits a barrier, and let $(S, T, U)$ be a maxmin barrier. For convenience, we restate the properties of $(S, T, U)$ from Lemma~\ref{lem:maxmin-barrier} for $k=3$.

\begin{claim}[Special case of Lemma~\ref{lem:maxmin-barrier}]\label{claim:maxmin-barrier-3-factor}
The following properties hold:
\begin{enumerate}[label=(\roman*)]
   \item For all vertices $w \in U$, $e(w, T) \le 2$; moreover, if $w$ lies in an even component of $G[U]$, then $e(w, T) \le 1$. 
   \item The maximum degree $\Delta(G[T]) \le 1$.
   \item For any odd component $C$ of $G[U]$ and any vertex $w \in T$, $e(w, T) + e(w, C) \le 2$.
   \item For any vertex $w \in T$, $e(w, T)+\sum_{C}e(w, C) \le 1$, where $C$ ranges over all even components of $G[U]$.
   \end{enumerate}
\end{claim}

By Lemma~\ref{lem:barrier-range}, we will have 
\begin{equation}\label{deltaSTU}
\delta_G(S,T,U) = 3|S|-3|T|+\sum_{v \in T} d_{G-S}(v)-q(S,T,U) \le -2 
\end{equation}
throughout this proof.  We also make the following observation:
\begin{observation}\label{Lambda2Q}
    If $Q \subset U$ is a set that contains at most one vertex from any component of $G[U]$, and $|Q| \ge 2$, then $\Lambda^2(Q) \subseteq S \cup T$, so $\beta^2(G) \geq 1$ implies $|Q| \leq |S|+|T|.$
\end{observation}

The next claim justifies applying the condition $\beta^2(G) \ge 1$ to the set $T$ or to any set containing $T$. 

\begin{claim}\label{claim:large-t-3-factor}
$|T| \ge 2$.
\end{claim}
\begin{proof}
Inequality~\ref{deltaSTU} implies $q(S, T, U) \geq 3|S| - 3|T| + \sum_{v \in T} d_{G-S}(v) + 2$.  If $|T|=0$, then we see $q(S,T,U) \geq 2$. Let $Q$ be a set containing one vertex from each odd component in $G[U]$; by Observation~\ref{Lambda2Q}, $|S|\ge|Q|=q(S,T,U)$, and Inequality~\ref{deltaSTU} yields $|S| \geq 3|S|-0+0+2$, a contradiction.  If $|T|=1$, let $x$ be the only vertex of $T$. By Lemma~\ref{lem:minimum-degree-with-k}, $d_G(x) \ge 3$, and therefore $d_{G-S}(x) \ge 3-|S|$. This reduces Inequality~\ref{deltaSTU} to $q(S, T, U) \geq 3|S| -3 + (3-|S|) + 2 = 2|S| + 2$. Once again,  $q(S,T,U) \ge 2$, and letting $Q$ be a set containing one vertex from each odd component of $G[U]$ and applying Observation~\ref{Lambda2Q} yields $q(S,T,U)\leq |S|+1$.  This again is a contradiction. Hence $|T|\geq 2.$
\end{proof}

\begin{claim}\label{claim:q>0-3-factor}
    $q(S,T,U) > 0$. 
\end{claim}
\begin{proof}
Suppose not. By Claim~\ref{claim:maxmin-barrier-3-factor}$(i)-(ii)$, $e(w, T) \le 1$ for all $w \notin S$, which means that $\Lambda^2(T) \subseteq S$. Because $\beta^2(G) \ge 1$, we must have $|S| \ge |T|$, contradicting Inequality~\ref{deltaSTU}.
\end{proof}

For the remainder of the base case, we further assume that $G$ is a split graph with a split partition $(X, Y)$. The pair $(X, Y)$ is not necessarily unique, since a vertex $x \in X$ adjacent to every vertex in $Y$ could be moved to $Y$. For concreteness, we assume that $(X,Y)$ is a pair where $X$ is independent, $Y$ is a clique, and $|X|$ is as small as possible: no vertex of $X$ is adjacent to all of $Y$. 

Finally, we assume that $G$ is edge-maximal with these properties: there is no graph $G'$ with $V(G') = V(G)$ and $E(G') \supsetneq E(G)$ such that $G'$ is also a split graph with no $3$-factor.

Under the assumptions above, we can say more about the structural properties of $(S,T,U)$:

\begin{claim}\label{claim:split-stu-structure}
The following statements hold:
\begin{enumerate}
\item[(a)] $S \subseteq Y$.
\item[(b)] $|T \cap Y| \le 2$, and $G[T]$ contains at most one edge.
\item[(c)] At most one component of $G[U]$ intersects $Y$; all others are isolated vertices.
\item[(d)] For any vertex $x \in T \cap X$, if no component of $G[U]$ intersects $Y$ or the component intersecting $Y$ is even, then $d_{G-S}(x) \le 1$; otherwise, $d_{G-S}(x) \le 2$.
\item[(e)] For a vertex $x \in U \cap X$, if no component of $G[U]$ intersects $Y$ or if $x$ is in such a component, then $e(x, T) \le |T \cap Y|-1$.
\end{enumerate}
\end{claim}
\begin{proof}
To prove (a), suppose for contradiction that $x \in S \cap X$, and let $H$ be the graph obtained from $G$ by adding all edges between $x$ and $Y$. The graph $H$ is a split graph with split partition $(X\setminus\{x\}, Y \cup \{x\})$. Also, $\delta_H(S,T,U) = \delta_G(S,T,U)$, which is negative by assumption, so by Tutte's $3$-factor theorem, $H$ does not have a $3$-factor. This contradicts either the assumption of edge-maximality of $G$ or our choice of $(X, Y)$.

To prove (b), first recall from Claim~\ref{claim:maxmin-barrier-3-factor}$(ii)$ that $G[T]$ has maximum degree $1$. Since $G[T \cap Y]$ is a clique, it can have at most $2$ vertices, or it exceeds this maximum degree. If $T \cap Y = \emptyset$, $T$ must be an independent set. If $T \cap Y = \{y\}$, then all edges in $G[T]$ must have $y$ as an endpoint; because $e(y, T) \le 1$, there can be at most one edge. Finally, if $T \cap Y = \{y_1, y_2\}$, then all edges in $G[T]$ must have $y_1$ or $y_2$ as an endpoint; because $e(y_1, T), e(y_2, T) \le 1$ and $y_1y_2$ is an edge in $G[T]$, there can be no other edges.

Part (c) holds because $G[U]$ is also a split graph; $G[U \cap Y]$ is a clique, so it is contained entirely in one component of $G[U]$. All other components lie in the independent set $X$ and are therefore isolated vertices.

To prove (d), let $x \in T\cap X$ be arbitrary. By part (c), there is at most one component of $G[U]$ intersecting $Y$; if such a component exists, denote it by $M$. Then $d_{G-S}(x) = e(x,T) + e(x,M)$, where $e(x,M) = 0$ if $M$ does not exist. By Claim~\ref{claim:maxmin-barrier-3-factor}$(ii)-(iv)$, we have $d_{G-S}(x) \le 1$ if $M$ is nonexistent or even, and $d_{G-S}(x) \le 2$ if $M$ is odd.
  
In part (e), since $x \in X$, all neighbors of $x$ in $T$ are in $T \cap Y$. Suppose for contradiction $x$ is adjacent to all of them. Then let $H$ be the graph obtained from $G$ by adding all edges from $x$ to $Y$. The graph $H$ is a split graph with split partition $(X\setminus \{x\}, Y \cup \{x\})$. Also, the new edges of $H$ are added only from $x$ to $S$ or to $x$'s component of $G[U]$, so $(S,T,U)$ is still a barrier in $H$, and $H$ does not have a $3$-factor. This contradicts either the assumption of edge-maximality of $G$ or our choice of $(X,Y)$.
\end{proof}

Out of all maxmin barriers, select the one with the minimum value of $q(S,T,U)$. If $U \cap Y \ne \emptyset$, we call the component of $G[U]$ containing the vertices of $Y$ the \emph{monster component} $M$. For the remainder of the proof, we consider two cases: there is a monster component, and there is no monster component.

\subsubsection{Monster Component}

Suppose that the monster component $M$ exists. 
\begin{claim}\label{allodd}
    All components of $G[U]$, including $M$, are odd.
\end{claim}
\begin{proof}
    If $M$ is an even component, then by Claim~\ref{claim:q>0-3-factor}, there must also be an odd component of $G[U]$, which consists of a single vertex $\{x\}$. Let $y \in V(M) \cap Y$. Then $H = G + xy$ satisfies $\delta_H(S,T,U) = \delta_G(S,T,U) < 0$, since in $H$ the even component $M$ and the odd component $\{x\}$ are replaced by a single odd component $M \cup \{x\}$. Thus $H$ has no $3$-factor. This contradicts the edge-maximality of $G$, so $M$ is odd. The same applies if $M$ is an odd component and $\{x\}$ is an even component. Therefore all components of $G[U]$ are odd.
\end{proof}



By Claim~\ref{claim:split-stu-structure}$(b)$, $T \cap Y$ contains at most two vertices, and we consider two cases for $|T \cap Y|$.

\textbf{Case 1: $T \cap Y = \emptyset$.}  Here, $T \cup (U \setminus V(M))$ is contained in $X$. We may assume $V(M) \subseteq Y$ by the same argument as in Claim~\ref{claim:split-stu-structure}$(a)$: if $x \in V(M) \cap X$, then adding all edges from $x$ to $Y$ and then moving $x$ from $X$ to $Y$ does not change the split graph structure or the deficiency of $(S,T,U)$, which contradicts either the edge-maximality of $G$ or the choice of $(X,Y)$. 

For all $y \in V(M)$, $e(y, T) \le 2$ by Claim~\ref{claim:maxmin-barrier-3-factor}$(i)$.
Symmetrically, for all $x \in T$, since $T \cup (U \setminus V(M))$ is independent, we have $d_{G-S}(x) = e(x, M)$, and by Claim~\ref{claim:split-stu-structure}$(d)$, $e(x, M) \le 2$. 
We conclude that the bipartite subgraph $G[T,M]$ has maximum degree $2$: it is a union of paths and cycles. 


\textbf{Case 1a: $M$ is not the only odd component.} In this case, we claim that $e(x, M) = 0$ for all $x \in T$, from which $\sum_{x \in T} d_{G-S}(x) = 0$ follows immediately.

Suppose not; let $x \in T$ with $e(x, M) \in \{1,2\}$. Choose $x' \in U\setminus V(M)$. Since $T\cap Y = \emptyset$, $d_{G-S}(x') = 0$. Define 
\[
(S, T', U') := \left(S, (T\setminus\{x\}) \cup \{x'\}, (U\setminus\{x'\}) \cup \{x\}\right),
\]
which is obtained from $(S, T, U)$ by swapping $x$ and $x'$. Then 
\begin{align*}
	\delta_G(S,T',U') &= \delta_G(S,T,U) + d_{G-S}(x') - d_{G-S}(x) + \Big(q(S,T,U) - q(S,T',U')\Big) \\
	&= \delta_G(S,T,U) - e(x, M) + \Big(q(S,T,U) - q(S,T',U')\Big).
\end{align*}
All odd components with respect to $(S,T,U)$ remain odd with respect to $(S,T',U')$ except $\{x'\}$ and possibly $M$; since $3|M \cup \{x\}| + e(M\cup \{x\}, T') = 3|M| + 3 + e(M,T) - e(x,M)$, $M$ remains odd if $e(x, M)=1$ and becomes even if $e(x, M)=2$. Therefore $q(S,T,U) - q(S,T',U')$ is $1$ if $e(x, M)=1$ and $2$ if $e(x, M)=2$; in either case, $\delta_G(S,T',U') = \delta_G(S,T,U)$. Therefore $(S,T',U')$ is also a barrier; it is a maxmin barrier, since $|T|=|T'|$, and $q(S,T',U') < q(S,T,U)$, contradicting the assumption that $(S, T, U)$ is a maxmin barrier minimizing $q$.

Inequality~\ref{deltaSTU} simplifies to $3|S| - 3|T| - q(S,T,U) \leq -2$. However, if we let $Q$ be a set of one vertex from each odd component of $G[U]$, then $|Q| \ge 2$ and $\Lambda^2(T \cup Q) \subseteq S$, so $|S| \ge |T| + |Q| = |T| + q(S,T,U)$. These two inequalities combine to give $2q(S,T,U) < -2$, which is a contradiction.

\textbf{Case 1b: $M$ is the only odd component.} In this case, Inequality~\ref{deltaSTU} implies that 
\[
    \delta_G(S,T,U) = 3|S| - 3|T| + e(T,M) - 1 \le -2,
\]
or $3|S| - 3|T| + e(T,M) < 0$. To show that this is impossible, we classify all possible components of $G[T,M]$, and their numbers. Define:
\begin{itemize}
\item $c_k$ to be the number of $2k$-vertex cycle components in $G[T,M]$ ($k \geq 2$).
\item $z_k$ to be the number of $2k$-vertex path components in $G[T,M]$ ($k \geq 1$).
\item $t_k$ to be the number of $(2k+1)$-vertex path components in $G[T,M]$ that have one more vertex in $T$ than in $M$ ($k \geq 0$; in particular, $t_0$ is the number of vertices in $T$ with no neighbors in $M$).
\item $m_k$ to be the number of $(2k+1)$-vertex path components in $G[T,M]$ that have one more vertex in $M$ than in $T$ ($k \geq 0$).
\end{itemize}
Let $T^*$ be a maximum subset of $T$ such that $\Lambda^2(T^*) \subseteq S$. Going component-by-component, we have
\begin{equation}
    |S| \ge |T^*| = \sum_{k\ge 2}\left\lfloor \frac k2\right\rfloor c_k + \sum_{k \ge 1} \left\lceil \frac k2\right\rceil z_k + \sum_{k \ge 0} \left\lceil \frac {k+1}2\right\rceil t_k + \sum_{k \ge 0} \left\lceil \frac k2\right\rceil m_k.\label{eq:t*-bound}
\end{equation}
Note that if $|T^*|=1$, then we cannot conclude $|S| \ge |T^*|$ from $\Lambda^2(T^*) \subseteq S$.
However, in this case, if $|S|<|T^*|$, then $S=\emptyset$. Since $|T| \geq 2$ and $T \subseteq X$ is independent, Claim~\ref{claim:split-stu-structure}$(d)$ implies that each vertex in $T$ has at most two neighbors in $G$, which contradicts Lemma~\ref{lem:minimum-degree-with-k}. So $|S| \ge |T^*|$.

Let $M^*$ be a maximum subset of $V(M)$ such that $\Lambda^2(T \cup M^*) \subseteq S \cup M$, so $|S| + |M| \ge |T| + |M^*|$. Going component by component to find each component's contribution to $|T| - |M| + |M^*|$, we have
\begin{equation}
    |S| \ge |T| - |M| + |M^*| = \sum_{k\ge 2}\left\lfloor \frac k2\right\rfloor c_k + \sum_{k \ge 1} \left\lceil \frac k2\right\rceil z_k + \sum_{k \ge 0} \left\lceil \frac k2 + 1\right\rceil t_k + \sum_{k \ge 0} \left\lceil \frac {k+1}2 - 1\right\rceil m_k.\label{eq:m*-bound}
\end{equation}
Finally, after rearranging $\delta_G(S,T,U) \le -2$ as $3|S| < 3|T| - e(T,M)$, we can go component-by-component to express $3|T| - e(T,M)$ in terms of $c_k$, $z_k$, $t_k$ and $m_k$, deducing that
\begin{equation}
    3|S| < \sum_{k \ge 2} kc_k + \sum_{k\ge 1} (k+1)z_k + \sum_{k\ge 0} (k+3)t_k + \sum_{k \ge 0} km_k.\label{eq:delta-bound}
\end{equation}
Taking $2 \cdot \eqref{eq:t*-bound} + \eqref{eq:m*-bound}$, and using the identity $\lceil \frac k2\rceil + \lceil \frac{k+1}{2}\rceil = k+1$ to simplify some of the terms, we get
\begin{equation}
    3|S| \ge \sum_{k\ge 2}3\left\lfloor \frac k2\right\rfloor c_k + \sum_{k \ge 1} 3\left\lceil \frac k2\right\rceil z_k + \sum_{k \ge 0} \left(\left\lceil \frac {k+1}2 \right\rceil + k + 2\right)t_k + \sum_{k \ge 0} \left(\left\lceil \frac k2\right\rceil + k\right)m_k.\label{eq:combined-bound}
\end{equation}
However, if we now compare \eqref{eq:delta-bound} and \eqref{eq:combined-bound} term-by-term, we notice that:
\begin{itemize}
    \item $3\lfloor \frac k2\rfloor \ge 1 + 2\lfloor \frac k2\rfloor \ge 1 + 2 \cdot \frac{k-1}{2} = k$ for $k\ge 2$;
    \item $3\lceil \frac k2\rceil = 3 > 2 = k+1$ for $k = 1$, and $3\lceil \frac k2\rceil \ge \frac{3k}{2} \ge k + 1$ for $k \ge 2$, so $3\lceil \frac k2\rceil \ge k+1$ for $k \ge 1$;
    \item $\lceil \frac {k+1}2 \rceil + k + 2 \ge 1 + k + 2 = k+3$ for $k\ge 0$;
    \item $\lceil \frac k2\rceil + k \ge k$ for $k \ge 0$.
\end{itemize}
Therefore the two inequalities are inconsistent, which rules out Case 1b.

\textbf{Case 2: $|T \cap Y| \ge 1$.} Let $t_X = |T \cap X|$, $t_Y = |T \cap Y|$, $m_X = |V(M) \cap X|$, and $m_Y = |V(M) \cap Y|$.  Of the non-monster components of $G[U]$, all of which are odd singleton components, let $q_0$ be the number with no neighbors in $T$ and let $q_2$ be the number with two neighbors in $T$. Then $t_Y \in \{1, 2\}$ by Claim~\ref{claim:split-stu-structure}$(b)$, and $m_Y \ge 1$ by the case assumption.

Compare $\delta_G(S,T,U)$ to $\delta_G(S,T',U')$ where $T' = T \cap X$ and $U' = U \cup (T \cap Y)$. The difference $\delta_G(S,T,U) - \delta_G(S,T',U')$ consists of
\begin{itemize}
\item $3|T'| - 3|T| = -3t_Y$;
\item $\sum_{v \in T \cap Y} d_{G-S}(v) \ge t_Y(t_Y + m_Y - 1) + 2q_2$;
\item $q(S,T',U') - q(S,T,U) \ge -q_2 - 1$, since at most $M$ and the odd singleton components with neighbors in $T$ can become even. 
\end{itemize}
If $t_Y(t_Y + m_Y - 4) + q_2 \ge 0$, then $\delta_G(S,T,U) - \delta_G(S,T',U') \ge -1$; since $\delta_G(S,T,U) \le -2$, $(S,T',U')$ is also a barrier, contradicting the assumption that $(S,T,U)$ is a maxmin barrier. Therefore $t_Y(t_Y + m_Y - 4) + q_2 < 0$. 

Note that $t_Y \in \{1,2\}$ and $m_Y \ge 1$; additionally, $q_2 > 0$ is only possible if $t_Y = 2$, since otherwise a vertex in $U \setminus V(M)$ cannot have two neighbors in $T$. This leaves four possibilities for $t_Y$, $m_Y$, and $q_2$:
\[
	(t_Y,m_Y,q_2) \in \Big\{ (1,1,0), (1,2,0), (2,1,0), (2,1,1) \Big\}.
\]
In all four cases, $t_Y(t_Y + m_Y - 4)$ is $-2$ or $-1$.

A lower bound on $\sum_{v \in T} d_{G-S}(v)$ is, once again, $2q_2 + t_Y(t_Y + m_Y - 1)$, and therefore
\[
	3|S| - 3t_X + t_Y(t_Y + m_Y - 4) + q_2 - q_0 - 1 \le \delta_G(S,T,U) \le -2.
\]
Since $t_Y(t_Y + m_Y - 4) \ge -2$, we can deduce the simpler inequality
\[
	3|S| - 3t_X + q_2 - q_0 \le 1.
\]

We now want to consider $\Lambda^2((T \cap X) \cup (U \setminus V(M)))$, but first we must deduce that $t_X + q_0 + q_2 \ge 2$. If not, then it follows from the inequality above that $|S| \le 1$, hence $|Y| = |S| + t_Y + m_Y \le 4$. Combined with Observation~\ref{obs:Y-range-split} for $k = 2$ and the assumption that $|G| = |X| + |Y| \ge 4$ is even, the only possible values for $(|X|, |Y|)$ are $(1, 3)$, $(0, 4)$, and $(2, 4)$. It is easy to check the possible graphs to see that all of them either contain a $3$-factor or have minimum degree less than $3$, violating Lemma~\ref{lem:minimum-degree-with-k}.


We can now draw several conclusions about $\Lambda^2((T \cap X) \cup (U\setminus V(M)))$:
\begin{itemize}
\item It contains no vertex of $V(M) \cap Y$, since these have no neighbors in $U\setminus V(M)$, and at most one neighbor in $T \cap X$ (they have at most two neighbors in $T$, by Claim~\ref{claim:maxmin-barrier-3-factor}$(i)$, and at least one is in $T \cap Y$).
\item If $q_2 = 0$, it contains no vertex of $T \cap Y$, since there are now no edges between $U\setminus V(M)$ and $T\cap Y$, by Claim~\ref{claim:split-stu-structure}$(b)$ there is at most one edge between $T \cap X$ and $T \cap Y$.
\item If $q_2 = 1$, it also contains no vertex of $T \cap Y$: in this case, $|T\cap Y|=2$, so the only edge in $G[T]$ lies within $T \cap Y$. Consequently, each vertex in $T \cap Y$ has exactly one neighbor in $U\setminus V(M)$ and none in $T \cap X$.
\end{itemize}
Therefore $\Lambda^2((T \cap X) \cup (U\setminus V(M))) \subseteq S$, and we conclude $|S| \ge |\Lambda^2((T \cap X) \cup (U\setminus V(M)))| \geq t_X + q_0 + q_2$. Substituting this into our earlier inequality, we get
\[
	3(t_X + q_0 + q_2) - 3t_X + q_2 - q_0 \le 1 \implies 2q_0 + 4q_2 \le 1
\]
hence $q_0 = q_2 = 0$ and $U = V(M)$. The inequalities $3|S| - 3t_X + q_2 - q_0 \le 1$ and $|S| \ge t_X + q_0 + q_2$ now imply, together, that $|S| = t_X$. Moreover, since $t_X+q_0+q_2 \ge 2$, it follows that $|S| = t_X \ge 2$.

Returning to the original inequality $\delta_G(S,T,U) \le -2$, we can refine the expression as follows:
\[
	3|S| - 3t_X + t_Y(t_Y + m_Y - 4) + 2e(T\cap X, T\cap Y) + e(T\cap X, V(M)\cap Y) + e(T\cap Y, V(M)\cap X) - 1 \le -2,
\]
since each edge between $T \cap X$ and $T \cap Y$ contributes $2$ to the degree sum in $\delta_G(S,T,U)$, while each edge between $T \cap X$ and $V(M) \cap Y$, or between $T \cap Y$ and $V(M) \cap X$, contributes only $1$ to the degree sum. Since $|S| = t_X$ and $t_Y(t_Y + m_Y - 4) \ge -2$, we have
\[
	2e(T \cap X, T \cap Y) + e(T \cap X, V(M) \cap Y) + e(T \cap Y, V(M) \cap X) \le 1.
\]
Hence there are no edges between $T \cap X$ and $T\cap Y$, and there is at most one edge in total between $T \cap X$ and $V(M) \cap Y$, and between $T \cap Y$ and $V(M) \cap X$.

Choose $y \in V(M) \cap Y$ with a neighbor in $T \cap X$ if such a vertex exists; otherwise, choose $y$ arbitrarily. Then $\Lambda^2((T \cap X) \cup \{y\}) \subseteq S$: apart from $y$, no vertex in $T \cup U$ has any edges to $T \cap X$. Since $|(T \cap X) \cup \{y\}| = t_X + 1$ and $|S| = t_X$, this violates the assumption that $\beta^2(G) \ge 1$.

\subsubsection{No Monster Component}

If there is no monster component, then $U$ is entirely contained in $X$. By Claim~\ref{claim:split-stu-structure}$(b)$, $|T \cap Y| \le 2$, and by Claim~\ref{claim:split-stu-structure}$(e)$, each vertex of $U$ can have at most one neighbor in $T$, which is in $T \cap Y$. Each vertex of $T \cap X$ can also have at most one neighbor in $V(G)\setminus S$, by Claim~\ref{claim:split-stu-structure}$(d)$, and it is also in $T \cap Y$. Let $Q$ be the set of vertices in odd components of $G[U]$. Then $\Lambda^2(T \cup Q) \subseteq S \cup (T \cap Y)$. 

Moreover, if $m = |T \cap Y \cap \Lambda^2(T \cup Q)|$, then $m \le 2$ and $\sum_{v \in T} d_{G-S}(v) \ge 2m$, because if $y \in T \cap Y \cap \Lambda^2(T \cup Q)$, then $d_{G-S}(y) \ge 2$. 

Applying the neighborhood condition of $\beta^2(G) \ge 1$ to $T \cup Q$, we conclude that 
\[
    |T| + q(S,T,U) = |T| + |Q| \le |\Lambda^2(T \cup Q)| \le |S| + m.
\]
But then substituting from this inequality for $|S|$ in Inequality~\ref{deltaSTU},
\begin{align*}
-2 \ge 3(|T| + q(S,T,U) - m) - 3|T| + 2m - q(S,T,U) 
    = 2q(S,T,U) - m.
\end{align*}
Since $m \le 2$, we have $2q(S,T,U) \le m - 2 \le 0$, so $q(S,T,U) = 0$, contradicting Claim~\ref{claim:q>0-3-factor}.

\subsection{The General Case}

\begin{proof}[Proof of Theorem~\ref{thm:k+1-factor-split-graph}]
Let $G$ be a split graph with even order $n \ge 2k$ and $\beta^k(G) \ge 1$ for some integer $k \ge 2$. By Theorem~\ref{thm:1-factors-nonbip}, the graph $G$ contains at least $k-1$ pairwise disjoint perfect matchings. Select any $k-2$ of them, and let $M$ denote their union. Then $M$ is a $(k-2)$-factor in $G$. Let $H = G-E(M)$. By Lemma~\ref{lem:f-dominating} with $F = M$, we have $\beta^2(H) \ge 1$. Together with $V(H) = V(G)$, this implies that $H$ is a split graph of even order with $\beta^2(H) \ge 1$. By the base case shown in the preceding section, $H$ admits a $3$-factor disjoint from $M$, and the union of this $3$-factor with $M$ is a $(k+1)$-factor in $G$.
\end{proof}

\section{Future Work}
We highlight two potential areas for future study. The graph $G$ consisting of a clique $K_{n-1}$ together with a vertex adjacent to exactly $2k - 1$ vertices of the clique satisfies $\beta^k(G) \ge 1$ but does not admit a $2k$-factor. By Lemma~\ref{lem:minimum-degree-with-k}, the degree $2k-1$ vertex cannot be replaced by a vertex of smaller degree. Together with the results of this paper, this leaves the existence of factors with intermediate degrees between $k + 1$ and $2k - 1$ for graphs satisfying $\beta^k(G) \ge 1$ open. We also speculate that graphs satisfying $\beta^k(G) \ge 1$ may admit more disjoint perfect or near-perfect matchings than currently known, particularly when additional graph properties are present.

We also believe that the results in Section~\ref{subsec:1-factors-bip} can be improved. For the weak bipartite binding number, we know $\beta^k(G, X) \ge 1$ ensures the existence of $k$ disjoint matchings, each covering all of $X$. When $k = 2$, a result of Bar\'at \textit{et al.}~\cite{Hungarian} implies that, under the same condition, $G$ admits a $2$-factor covering $X$. We conjecture that $\beta^k(G,X) \geq 1$ implies the existence of a $k$-factor covering $X$.


\section*{Conflict of Interest Statement}
The authors declare that they have no conflict interest regarding
the publication of this paper.

\bibliographystyle{plain}
\bibliography{Reference.bib} 

\begin{thebibliography}{10}

\bibitem{ANDERSON1971}
I.~Anderson.
\newblock Perfect matchings of a graph.
\newblock {\em Journal of Combinatorial Theory, Series B}, 10(3):183--186, 1971.

\bibitem{Hungarian}
J.~Bar\'{a}t, A.~Grzesik, A.~Jung, Z.~L\'{o}r\'{a}nt Nagy, and D.~P\'{a}lv\"{o}lgyi.
\newblock The double {H}all property and cycle covers in bipartite graphs.
\newblock {\em Discrete Math.}, 347(9):114079, 2024.

\bibitem{CLMSV2025}
G.~Chen, M.~Lavrov, Y.~Ma, Y.~Su, and J.~Vandenbussche.
\newblock Bipartite graphs with the double {H}all property, 2025.
\newblock \href{https://arxiv.org/abs/2502.10903}{arXiv:2502.10903}.

\bibitem{diestel}
R.~Diestel.
\newblock {\em Graph Theory}.
\newblock Springer, 2005.

\bibitem{GRIMM2023}
E.~Grimm, A.~Johnsen, and S.~Shan.
\newblock Existence of 2-factors in tough graphs without forbidden subgraphs.
\newblock {\em Discrete Mathematics}, 346(10):113578, 2023.

\bibitem{Hu2013}
Z.~Hu, K.~H. Law, and W.~Zang.
\newblock An optimal binding number condition for bipancyclism.
\newblock {\em SIAM Journal on Discrete Mathematics}, 27(2):597--618, 2013.

\bibitem{kannoshan2019}
J.~Kanno and S.~Shan.
\newblock {V}izing's 2-factor conjecture involving toughness and maximum degree conditions.
\newblock {\em The Electronic Journal of Combinatorics}, pages P2--17, 2019.

\bibitem{KATERINISWOODALL}
P.~Katerinis and D.~R. Woodall.
\newblock Binding numbers of graphs and the existence of k-factors.
\newblock {\em The Quarterly Journal of Mathematics}, 38(2):221--228, 06 1987.

\bibitem{LEBENSOLD1977}
K.~Lebensold.
\newblock Disjoint matchings of graphs.
\newblock {\em Journal of Combinatorial Theory, Series B}, 22(3):207--210, 1977.

\bibitem{Petersen1891}
J.~Petersen.
\newblock Die {T}heorie der regulären graphs.
\newblock {\em Acta Mathematica}, 15(none):193 -- 220, 1900.

\bibitem{Qian2001}
J.~Qian.
\newblock Two sufficient conditions for the existence of \textit{k}-factors in bipartite graphs.
\newblock {\em Journal of Shandong University (Natural Science Edition)}, 36(4):477--480, 2001.

\bibitem{Tutte1954ASP}
W.~T. Tutte.
\newblock A short proof of the factor theorem for finite graphs.
\newblock {\em Canadian Journal of Mathematics}, 6:347 -- 352, 1954.

\bibitem{WOODALL1973}
D.~R. Woodall.
\newblock The binding number of a graph and its {A}nderson number.
\newblock {\em Journal of Combinatorial Theory, Series B}, 15(3):225--255, 1973.

\end{thebibliography}

\end{document}